\documentclass[12pt]{amsart}

\usepackage{frcursive}
\usepackage[T1]{fontenc}
\newcommand{\dd}{ {\mbox{\sl{\cursive{d}\,}}} }
\newcommand{\dr}{ {\mbox{\sl{\cursive{d}\,}}} }

\usepackage{color}
\usepackage{amsmath, amsthm, amssymb}
\usepackage{amsfonts}
\usepackage[ansinew]{inputenc}
\usepackage{graphicx}
\usepackage[english]{babel}
\usepackage{hyperref}
\theoremstyle{plain}
\newtheorem{thm}{Theorem}[section]
\newtheorem{cor}[thm]{Corollary}
\newtheorem{lem}[thm]{Lemma}
\newtheorem{prop}[thm]{Proposition}

\theoremstyle{definition}

\theoremstyle{remark}
\newtheorem{rem}[thm]{Remark}

\numberwithin{equation}{section}

\newcommand{\average}{{\mathchoice {\kern1ex\vcenter{\hrule height.4pt
width 6pt depth0pt} \kern-9.7pt} {\kern1ex\vcenter{\hrule
height.4pt width 4.3pt depth0pt} \kern-7pt} {} {} }}

\def\R{\mathbb{R}}
\def\Q{\mathbb{Q}}

\def\N{\mathbb{N}}

\title[The Dirichlet problem for nonlocal operators]{The Dirichlet
problem \\
for nonlocal operators with singular kernels: \\
convex and nonconvex domains}

\author{Xavier Ros-Oton}
\address{The University of Texas at Austin,
Department of Mathematics, 2515 Speedway, Austin, TX 78751, USA}
\email{ros.oton@math.utexas.edu}

\author{Enrico Valdinoci}
\address{Weierstrass Institut f\"ur Angewandte Analysis
und Stochastik, Mohrenstrasse 39, 10117 Berlin, Germany,
and Dipartimento di Matematica, Universit\`a degli studi di Milano,
Via Saldini 50, 20133 Milan, Italy, and
Istituto di Matematica Applicata e Tecnologie Informatiche,
Consiglio Nazionale delle Ricerche,
Via Ferrata 1, 27100 Pavia, Italy}
\email{enrico@mat.uniroma3.it}



\begin{document}

\begin{abstract}
We study the interior regularity of solutions to the Dirichlet problem $Lu=g$ in $\Omega$,
$u=0$ in $\R^n\setminus\Omega$, for anisotropic operators of fractional type
$$ Lu(x)= \int_{0}^{+\infty}\,d\rho
\int_{S^{n-1}}\,da(\omega)\,
 \frac{
2u(x)-u(x+\rho\omega)-u(x-\rho\omega)}{\rho^{1+2s}}.$$
Here, $a$ is any measure on~$S^{n-1}$ (a prototype example
for~$L$ is given
by the sum of one-dimensional fractional Laplacians in fixed,
given directions).

When $a\in C^\infty(S^{n-1})$ and $g$ is $C^\infty(\Omega)$, solutions are known to be $C^\infty$
inside~$\Omega$ (but not up to the boundary).
However, when $a$ is a general measure, or even when $a$ is $L^\infty(S^{n-1})$,
solutions are only known to be $C^{3s}$ inside $\Omega$.

We prove here that, for general measures $a$, solutions are $C^{1+3s-\epsilon}$
inside $\Omega$ for all $\epsilon>0$ whenever $\Omega$ is convex.
When $a\in L^{\infty}(S^{n-1})$, we show that the same holds in all $C^{1,1}$ domains.
In particular, solutions always possess a classical first derivative.

The assumptions on the domain are sharp,
since if the domain is not convex and the measure $a$
is singular, we construct an explicit counterexample for which $u$ is
\emph{not} $C^{3s+\epsilon}$ for any $\epsilon>0$ -- even if $g$ and $\Omega$ are $C^\infty$.
\end{abstract}

\maketitle

\section{Introduction}

Recently, a great attention in the literature has been devoted
to the study of equations of elliptic type
with fractional order $2s$, with $s\in(0,1)$.
The leading example of
the operators considered is the fractional Laplacian
\begin{equation}\label{L22}
(-\Delta)^s u(x)= \int_{\R^n}
\frac{2u(x)-u(x+y)-u(x-y)}{|y|^{n+2s}}dy.
\end{equation}
Several similarities arise between this operator and
the classical Laplacian: for
instance, the fractional Laplacian enjoys
a ``good'' interior regularity theory in H\"older spaces
and in Sobolev spaces (see e.g.~\cite{land}).
Nevertheless,
the fractional Laplacian also presents some striking difference
with respect to the classical case: for example,
solutions are in general not uniformly Lipschitz continuous
up to the boundary (see e.g.~\cite{getoor, SRB})
and fractional harmonic functions are locally dense in~$C^k$
(see~\cite{dense}), in sharp
contrast with respect to the classical case.\medskip

A simple difference between the fractional and the classical Laplacians
is also given by the fact that the classical Laplacian
may be reconstructed as the sum of finitely many one-dimensional
operators, namely one can write
\begin{equation}\label{PL}
\Delta =
\partial^2_{1}+\cdots+\partial^2_{n},\end{equation}
and each~$\partial^2_{i}$ is indeed the Laplacian in a given direction.
This phenomenon is typical for the classical case
and it has no counterpart in the fractional setting,
since the operator in~\eqref{L22} cannot be reduced
to a finite sets of directions.\medskip

Nevertheless, in order to study equations
in anisotropic media,
it is important to understand operators
obtained by the superposition of fractional one-dimensional (or
lower-dimensional) operators, or, more generally, by
the superposition of different operators in different
directions, see~\cite{SROP}.
For this reason,
we consider here the anisotropic integro-differential operator
\begin{equation}\label{L}
Lu(x)= \int_{0}^{+\infty}\,d\rho
\int_{S^{n-1}}\,da(\omega)\,
 \frac{
2u(x)-u(x+\rho\omega)-u(x-\rho\omega)}{\rho^{1+2s}},
\end{equation}
with $s\in(0,1)$.
Here $a$ is a non-negative measure on~$S^{n-1}$
(called in jargon the ``spectral measure''), and we suppose that
it satisfies the following ``ellipticity'' assumption
$$ \inf_{\varpi\in S^{n-1}} \int_{S^{n-1}}|\omega\cdot\varpi|^{2s}\,da(\omega)\ge\lambda
\;{\mbox{ and }}\;
\int_{S^{n-1}}\,da \le \Lambda,$$
for some~$\lambda$, $\Lambda>0$.
The simplest model example is when $a$ is absolutely
continuous with respect to the standard measure on~$S^{n-1}$
(that is~$da(\omega)=a(\omega)\,d{\mathcal{H}}^{n-1}(\omega)$,
for a suitable $L^1$ function~$a:S^{n-1}\to[0,+\infty]$).
In this
case, thanks to the polar coordinate representation,
the operator $L$ may be written (up to a multiplicative
constant) as
\begin{equation}\label{L2}
Lu(x)= \int_{\R^n}\bigl(2u(x)-u(x+y)-u(x-y)\bigr)
\frac{a\left(y/|y|\right)}{|y|^{n+2s}}dy.
\end{equation}
When~$a\equiv1$ in~\eqref{L2}
(i.e. when~$da\equiv d{\mathcal{H}}^{n-1}$
in~\eqref{L}), we have the particularly famous case of
the fractional Laplacian in~\eqref{L22}.

In general, the role of the measure~$a$ in~\eqref{L}
is to weight differently
the different spacial directions (hence we refer to it
as an ``anisotropy'').
In particular, we can also allow the measure~$a$ in~\eqref{L}
to be a sum of Dirac's Deltas.
Indeed, a quite stimulating example arises in the case in which
\begin{equation}\label{8isdfv-0bg}
a=\sum_{i=1}^n \delta_{e_i} +\delta_{-e_i},
\end{equation}
where, as usual, $\{e_1,\cdots,e_n\}$ is the standard Euclidean base
of~$\R^n$: then the operator in~\eqref{L}
becomes
\begin{equation}\label{RI}
(-\partial^2_{1})^s+\cdots+(-\partial^2_{n})^s,
\end{equation}
where~$(-\partial^2_{i})^s$ represents the one-dimensional
fractional Laplacian in the $i$th coordinate direction
(compare with~\eqref{PL}).

Notice that, in the limit case $s=1$, the operators \eqref{L} that we study are linear elliptic and translation invariant, and hence they are of the form $Lu=-\sum_{i,j}a_{ij}\partial_{ij}u$.
After an affine change of variables, the operator $L$ is just the Laplacian $-\Delta$.
Thus, for nonlocal equations $s\in(0,1)$, the class of linear and translation invariant operators
is much richer than in the local case, and presents several interesting features that are purely nonlocal.

Goal of this paper is to develop a regularity theory for solutions
of
\begin{equation}\label{eq}
\left\{ \begin{array}{rcll}
L u &=&g&\textrm{in }\Omega \\
u&=&0&\textrm{in }\R^n\backslash\Omega,
\end{array}\right.
\end{equation}
The class of solutions that we study are the weak ones, i.e.
the ones that have finite (weighted) energy
$$ \int_{\R^n}\,dx \int_\R \,d\rho
\int_{S^{n-1}}\,da(\omega)\,
 \frac{
\big( u(x)-u(x+\rho\omega)\big)^2
}{\rho^{1+2s}}\,<\,+\infty$$
and satisfy
\begin{eqnarray*}&& \frac12 \int_{\R^n}\,dx \int_\R \,d\rho
\int_{S^{n-1}}\,da(\omega)\,
 \frac{
\big( u(x)-u(x+\rho\omega)\big)\,
\big( \varphi(x)-\varphi(x+\rho\omega)\big)}{\rho^{1+2s}}\\ &&\qquad\,=\,
\int_{\R^n}\,dx \, g(x)\,\varphi(x),\end{eqnarray*}
for any~$\varphi\in C^\infty_0(\Omega)$.\medskip

When the nonlinearity~$g$ is regular enough
and the measure $a$ of the operator~$L$ is $C^\infty(S^{n-1})$
it is known that solutions of~\eqref{eq} in bounded domains
are, roughly speaking, smooth up to
an additional order $2s$ in the derivatives: i.e. for any~$\beta\in[0,+\infty)$
such that~$\beta+2s$ is not an integer, we have that~$u\in C^{\beta+2s}_{\rm loc}(\Omega)$,
thanks to the estimate
\begin{equation}\label{REH}
\|u\|_{C^{\beta+2s}(B_{r/2})}
\le C\, \bigl( \|g\|_{C^\beta(B_r)} + \|u\|_{L^\infty(\R^n)}\bigr),
\end{equation}
valid in every ball $B_r$ in $\Omega$;
see for instance Corollary 3.5 in \cite{SROP} and also \cite{silvestre,BFV}.
The constant $C$ in \eqref{REH} depends on $n$, $s$, $r$,
and the $C^{\beta}(S^{n-1})$ norm of the measure $a$ (the ``anisotropy'').
When $\beta+2s$ is an integer, then the same estimate holds with $\|u\|_{C^{\beta+2s-\epsilon}}$ (for any $\epsilon>0$) in the left hand side of \eqref{REH}.

In particular, solutions of~\eqref{eq}
are $C^\infty(\Omega)$ if so is~$g$ and the measure $a$,
but in general they are not better than~$C^s(\R^n)$,
i.e. they are smooth in the interior, but only H\"older continuous
at the boundary.
For instance, $(-\Delta)^s (1-|x|^2)_+^s$
is constant in~$B_1$ and provides an example of solution
which is not better than~$C^s(\R^n)$.\medskip

In the general case of operators as in~\eqref{L},
the situation becomes quite different, due to the lack of
regularity of the kernels outside the origin.
In this generality,
estimate~\eqref{REH} does not hold, and it gets replaced by
the weaker estimate
\begin{equation}\label{REH2}
\|u\|_{C^{\beta+2s}(B_{r/2})}
\le C\big(
\|g\|_{C^\beta(B_r)}+\|u\|_{C^\beta(\R^n)}
\big) ,
\end{equation}
see Theorem 1.1 in \cite{SROP}.
The constant $C$ in \eqref{REH2} depends on the measure $a$ only through $\lambda$ and $\Lambda$.
\medskip

Though estimates~\eqref{REH} and~\eqref{REH2} may look
similar at a first glance, the additional term~$\|u\|_{C^\beta(\R^n)}$
in~\eqref{REH2} prevents higher regularity results:
namely, since~$u$ is {\em not} in general~$C^\beta(\R^n)$ when~$\beta>s$,
it follows that~\eqref{REH2} is meaningful mainly when $\beta\le s$
and it cannot provide higher order regularity: namely, from it one can only
show that~$u\in{C^{3s}_{\rm loc}(\Omega)}$, even if one assumes~$g\in C^\infty(\overline\Omega)$.
\medskip

In the light of these observations, in general, when~$s<1/3$, one does
not have any control even on the first derivative of~$u$. Nevertheless,
we will prove here a higher regularity result as in~\eqref{REH}, up to an exponent larger than one, by
relating the
differentiability properties of the solution with the geometry of the
domain. Namely, we will show that in convex domains and for $g$ smooth enough the solution is
always~$C^{1+3s-\epsilon}_{\rm loc}(\Omega)$, for any measure $a$.

The same regularity result holds
true also in possibly non-convex domains with~$C^{1,1}$ boundary,
provided that the measure~$a$ is bounded, i.e.
if~$da(\omega)=a(\omega)\,d{\mathcal{H}}^{n-1}(\omega)$,
with~$a\in L^\infty(S^{n-1})$.

In further detail, the main result that we prove is the following:

\begin{thm}\label{MAIN}
Let~$\beta\in (0,1+s)$ and assume that~$\beta+2s$ is not an integer.
Assume that either
\begin{equation}\label{HYP1}
{\mbox{$\Omega$ is a convex, bounded domain,}}
\end{equation}
or
\begin{equation}
\label{HYP2}\begin{split}
& {\mbox{$\Omega$ is a bounded domain with~$C^{1,1}$ boundary}}\\
& {\mbox{and the measure $a$ in \eqref{L} is bounded.}}
\end{split}\end{equation}
Let $u$ be a weak solution to \eqref{eq}, with~$g\in C^\beta(\overline\Omega)$.

Then~$u\in C^{\beta+2s}_{\rm loc}(\Omega)$
and, for any~$\delta>0$,
\begin{equation*}
\|u\|_{C^{\beta+2s}(\Omega_\delta)}\le C\,
\|g\|_{C^\beta(\overline\Omega)},\qquad \beta\in(0,1+s)
\end{equation*}
where~$\Omega_\delta$ is the set containing all the points
in~$\Omega$ that have distance larger than~$\delta$ from~$\partial\Omega$
and $C>0$ depends also on~$\Omega$ and~$\delta$.
\end{thm}

We think that it is a very interesting open problem
to establish whether or not a higher regularity theory
holds true under the assumptions of Theorem~\ref{MAIN}
(for instance, it is an open question to establish if
solutions are~$C^\infty$ if so are the data,
or if the $C^{1+3s}$ regularity is optimal also in this case).
As far as we know, there are natural examples
of smooth solutions (such as the one presented in Lemma~\ref{phi}),
but a complete regularity theory only holds under additional
regularity assumptions on the ``anisotropy'' $a$
(see~\cite{silvestre,BFV,SROP}) and the general picture seems to
be completely open.

\begin{rem}
Notice that under the hipotheses of Theorem \ref{MAIN} any weak solution to \eqref{eq} is bounded and classical, in the sense that they are $C^{2s+\epsilon}$
in the interior of the domain (see e.g. \cite[Section 5]{ROsurvey}).
Thus, in the proof of our results we may use the pointwise definition of the operator $L$.
\end{rem}

The result of Theorem~\ref{MAIN} also plays an important
role in the proof of a Pohozaev type identity
for anisotropic fractional operators, see~\cite{SROV}.
\medskip

At a first glance, it may also look surprising that the regularity
of the solution in Theorem~\ref{MAIN}
depends on the shape of the domain.
But indeed the convexity assumption
in Theorem~\ref{MAIN} cannot in general
be avoided, as next result points out:

\begin{thm}\label{MAIN2}
Let~$L$ be as in~\eqref{RI}, with~$n=2$.
There exists
a bounded domain in~$\R^2$ with $C^\infty$ boundary and a solution~$u\in C^s(\overline\Omega)$ of
\begin{equation*}
\left\{ \begin{array}{rcll}
L u &=&1&\textrm{in }\Omega \\
u&=&0&\textrm{in }\R^n\backslash\Omega,
\end{array}\right.
\end{equation*}
with~$u\not\in C^{3s+\epsilon}_{\rm loc}(\Omega)$, for any~$\epsilon>0$.
\end{thm}

We remark that the result in Theorem~\ref{MAIN2}
is special for the case of singular measures $a$
(compare, for instance, with the regularity results
in case of smooth measures in~\cite{silvestre,BFV,SROP}).
In particular, the loss of interior regularity detected by
Theorem~\ref{MAIN2} is in sharp
contrast with the smooth interior regularity theory that holds
true for both the classical and the fractional Laplacian.\medskip

Roughly speaking, the counterexample in Theorem~\ref{MAIN2}
will be based on the
fact that, if the domain is not convex, there are half-lines
originating from an interior point that intersect tangentially the boundary
of~$\Omega$: then, the singularity on $\partial\Omega$ (created by the solution itself)
``propagates'' inside the domain due to the nonlocal effect of the operator.
\medskip

The rest of the paper is organized as follows.
In Section~\ref{WN:SEC}
we recall the appropriate notion of weighted norms
that we will use to prove Theorem~\ref{MAIN}:
these norms are slightly unpleasant from the typographic point
of view, but they have nice scaling properties
and they encode the ``right'' behavior of the functions in
the vicinity of the boundary as well.

Then, Sections~\ref{S:2a} and~\ref{S:2b}
comprise several integral computations
of geometric flavor to estimate suitably averaged weighted
distance functions in the domain into consideration.
More precisely, Section~\ref{S:2a} is devoted to the
case of convex domains. The estimates obtained there
will be used for the proof of Theorem~\ref{MAIN} in case of convex domains,
where no structural assumption on the operator~$L$ is taken
and therefore the integrals considered are ``line
integrals'', as in~\eqref{L}.
Section~\ref{S:2b} is instead devoted to the case of
bounded domains with~$C^{1,1}$ boundary. The estimates
of this part will be used
in the proof of Theorem~\ref{MAIN} for~$C^{1,1}$ domains: since in
this case the operator~$L$ is
as in~\eqref{L2}, the integrals considered
are ``spread'' over~$\R^n$.
That is, roughly speaking, in Section~\ref{S:2a} the
singularity of the line integrals is compensated by the convexity
of the domain, while in Section~\ref{S:2b}
is the operator~$L$ that somehow
averages its effect in open regions of~$\R^n$.

In Section~\ref{US} we compute the effect of a cutoff
on the operator. Namely, when proving regularity results,
it is often useful to distinguish the interior regularity
from the one at the boundary (though, as shown here, in the nonlocal
setting one may dramatically interact with the other).
To this goal, it is sometimes desirable to localize the
solution inside the domain by multiplication with a cutoff function:
by performing this operation,
some estimates are needed in order to control the
effect of this cutoff on the operator.
These estimates, in our case, are provided in
Lemma~\ref{WW}.

In Section~\ref{IT:SEC}
we bootstrap the regularity theory obtained
in order to increase, roughly speaking, the H\"older exponent by~$2s$.
In our framework,
the model for such ``improvement of regularity'' result is given
by Theorem~\ref{CS}, which somehow allows us to say that
solutions in~$C^\alpha_{\rm loc}$ are in fact
in~$C^{\alpha+2s}_{\rm loc}$, if the nonlinearity is nice enough and~$\alpha<1+s$
(the precise statement involves weighted norms).
Section~\ref{IT:SEC}
contains also Corollary~\ref{IT},
which is the iterative version of Theorem~\ref{CS}
and provides
a very general regularity result, which in turn implies
Theorem~\ref{MAIN} (as a matter
of fact, in Corollary~\ref{IT}
it is not necessary to assume that~$g$ is~$C^\beta$
up to the boundary, but only that has finite weighted norm,
and also the weighted norm of~$u$ is controlled
up to the boundary).

The proof of Theorem~\ref{MAIN2} is contained in Section~\ref{CO:CO:S},
where a somehow surprising counterexample will be constructed.

The paper ends with an appendix, which collects some ``elementary'',
probably well known, but not trivial, ancillary results on the
distance functions (in our setting, these results are
needed for the integral computations of Section~\ref{S:2b}).

\section{Regularity framework with weighted norms}\label{WN:SEC}

To study the regularity theory up to the boundary, it is
convenient to
use the following notation for weighted norms.
We consider the distance from a point~$x\in\Omega$ to~$\partial\Omega$,
defined, as usual as
$$ \dd(x)={\rm dist}\,(x,\partial\Omega)=\inf_{q\in\partial\Omega} |x-q|.$$
We also denote
\begin{equation} \label{dxy}\dd(x,y)=\min\{ \dd(x),\,\dd(y)\}.\end{equation}
Of course, no confusion should arise
between~$\dd(x,y)$ and the distance function from $x$ to~$y$
(since the latter is denoted by~${\rm dist}\,(x,y)$).
Given~$\sigma\in\R$
and~$\alpha>0$, we take~$k\in\N$ and~$\alpha'\in(0,1]$
such that~$\alpha=k+\alpha'$ and we let
\begin{equation}\label{GN}\begin{split}
& [u]_{\alpha;\Omega}^{(\sigma)}
= \sup_{x\ne y\in\Omega} \left[
{ \dd^{\alpha+\sigma}(x,y)
\frac{|D^ku(x)-D^ku(y)|}{|x-y|^{\alpha'}} } \right]\\
{\mbox{and }}\;&\|u\|_{\alpha;\Omega}^{(\sigma)}=
\sum_{j=0}^k \sup_{x\in\Omega} \left[ \dd^{j+\sigma}(x)\,|D^j u(x)|\right]
+[u]_{\alpha;\Omega}^{(\sigma)}
.\end{split}\end{equation}
Notice that in~\eqref{GN} the number~$k$ is a ``sort of
integer part'' of~$\alpha$: more precisely,
when~$\alpha'\in(0,1)$ then~$k$ is 
integer part of~$\alpha$, but when~$\alpha'=1$ then~$\alpha$
is an integer and~$k=\alpha-1$. Related
notations concerning this type of modified integer part function
appear in the theory of Besov spaces.

The advantage of these weighted norms is twofold.
First of all,
since we write~$\alpha=k+\alpha'$ with~$\alpha'\in(0,1]$,
we can comprise the usual H\"older and Lipschitz
spaces~$C^\beta$, $C^{1+\beta}$, $C^{2+\beta}$, etc., with~$\beta\in(0,1]$
with the same notation.
As usual, given~$k\in\N$ and~$\beta\in(0,1]$, we define
\begin{equation}\label{RE:VB67}
\| u \|_{C^{k+\beta}(\Omega)}
:=\sum_{{\gamma\in\N^n}\atop{|\gamma|\le k}}
\sup_{x\in\Omega} |D^\gamma u(x)|
+\sum_{{\gamma\in\N^n}\atop{|\gamma|=k}}
\sup_{ {x,y\in\Omega}\atop{x\ne y} }
\frac{|D^\gamma u(x)-D^\gamma u(y)|}{|x-y|^\beta}.
\end{equation}
Notice that when~$-\sigma=\alpha\in(0,1]$,
the notation~$[u]_{\alpha;\Omega}^{(\sigma)}$ boils down to
the usual seminorm of~$C^\alpha(\Omega)$.
What is more, by choosing~$\sigma$
in the appropriate way, we can allow the derivatives of~$u$
to possibly blow up near the boundary, hence interior and boundary
regularity can be proved at the same time and interplay\footnote{Though not explicitly used in this paper, we remark that an additional
advantage of these weighted norms is that they usually behave
nicely for semilinear equations, namely when the nonlinearity in~\eqref{eq}
has the form~$g(x)=f(x,u(x))$,
in which~$f$ is locally
Lipschitz, but~$u$ is not.}
the one with the other.

The weighted norms
in~\eqref{GN} enjoy a monotonicity property with respect to~$\alpha$,
that is if~$\alpha_1\le\alpha_2$
and~$\|u\|_{\alpha_2;\Omega}^{(\sigma)}<+\infty$ then
also~$\|u\|_{\alpha_1;\Omega}^{(\sigma)}<+\infty$.
This is given by the following:

\begin{lem}\label{NORMS}
Let~$\alpha_1\le\alpha_2$. Then~$\|u\|_{\alpha_1;\Omega}^{(\sigma)}\le
C\,\|u\|_{\alpha_2;\Omega}^{(\sigma)}$, for some~$C>0$
only depending on~$\alpha_1$ and~$\alpha_2$
(and bounded uniformly when~$\alpha_1$ and~$\alpha_2$
range in a bounded set).
\end{lem}

\begin{proof} We write~$\alpha_i=k_i+\alpha'_i$, for~$i\in\{1,2\}$,
$k_i\in\N$ and~$\alpha'_i\in(0,1]$. We claim that
\begin{equation}\label{k}
k_1\le k_2.
\end{equation}
To prove it, suppose the converse: then~$k_1>k_2$
and therefore, being~$k_1$ and~$k_2$ integers, it follows that~$k_1\ge
k_2+1$. Then
$$ \alpha_2=k_2+\alpha'_2\le k_1+\alpha'_2-1
\le k_1<k_1+\alpha'_1=\alpha_1,$$
in contradiction with our assumptions. This proves~\eqref{k}.

Now we show that
\begin{equation}\label{ma89}
\dd^{\alpha_1+\sigma}(x,y)
\frac{|D^{k_1} u(x)-D^{k_1} u(y)|}{|x-y|^{\alpha'_1}} \le
C\,\|u\|_{\alpha_2;\Omega}^{(\sigma)}.
\end{equation}
Notice that, to prove~\eqref{ma89}, we can suppose that
\begin{equation}\label{ma89.1}
|x-y|\le \frac{\dd(x,y)}{4}.
\end{equation}
Indeed, if~$|x-y|>\dd(x,y)/4$ we use~\eqref{k} to see that
$$ \dd^{k_1+\sigma} (x) |D^{k_1} u(x)|\le\|u\|_{\alpha_2;\Omega}^{(\sigma)}$$
and therefore
\begin{eqnarray*}
&&\dd^{\alpha_1+\sigma}(x,y)
\frac{|D^{k_1} u(x)-D^{k_1} u(y)|}{|x-y|^{\alpha'_1} }\le
C\dd^{\alpha_1+\sigma}(x,y)
\frac{|D^{k_1} u(x)|+|D^{k_1} u(y)|}{\dd^{\alpha'_1}(x,y) }\\
&&\qquad\le
C\dd^{k_1+\sigma}(x,y)\,\big(
\dd^{-k_1-\sigma}(x)+\dd^{-k_1-\sigma}(y)
\big)\,\|u\|_{\alpha_2;\Omega}^{(\sigma)}
\le C\,\|u\|_{\alpha_2;\Omega}^{(\sigma)},\end{eqnarray*}
that shows~\eqref{ma89}
in this case. Hence, we reduce to prove~\eqref{ma89}
under the additional assumption~\eqref{ma89.1}.
To this goal, we observe that~\eqref{ma89.1} implies that
\begin{equation}\label{9sdfgf1qsssa12}
|\dd(x)-\dd(y)|\le |x-y|\le \frac{\dd(x,y)}{4}.\end{equation}
Now,
in view of~\eqref{k}, we can distinguish two cases:
either~$k_1=k_2$ or~$k_1<k_2$.
If~$k_1=k_2$, then
$$ \alpha_1' =\alpha_1-k_1=\alpha_1-k_2\le \alpha_2-k_2=\alpha'_2,$$
thus we set~$\dd_x=\dd(x)$ and~$\dd_{x,y}=\dd(x,y)$
for typographical convenience and we perform the following
computation:
\begin{eqnarray*}
&&
\dd^{\alpha_1+\sigma}_{x,y}
\frac{|D^{k_1} u(x)-D^{k_1} u(y)|}{|x-y|^{\alpha'_1} }\\
&=&\dd_{x,y}^{\alpha_1+\sigma}
\frac{|D^{k_1} u(x)-D^{k_1} u(y)|^{\frac{\alpha'_2-\alpha'_1}{\alpha'_2}}
|D^{k_1} u(x)-D^{k_1} u(y)|^{\frac{\alpha'_1}{\alpha'_2}}
}{|x-y|^{\alpha'_1} }\\
&\le& \dd_{x,y}^{\alpha_1+\sigma}
\,\dd_{x,y}^{\frac{-\alpha'_1(\alpha_2+\sigma)}{\alpha'_2}}
\frac{\big( |D^{k_1} u(x)|+|D^{k_1} u(y)|
\big)^{\frac{\alpha'_2-\alpha'_1}{\alpha'_2}}\,\big(
\dd_{x,y}^{\alpha_2+\sigma}
|D^{k_1} u(x)-D^{k_1} u(y)|\big)^{\frac{\alpha'_1}{\alpha'_2}}
}{|x-y|^{\alpha'_1} }\\
&\le&\dd_{x,y}^{\alpha_1+\sigma}
\dd_{x,y}^{\frac{-\alpha'_1(\alpha_2+\sigma)}{\alpha'_2}}
\frac{\big( \dd_x^{-k_1-\sigma}
\|u\|_{\alpha_2;\Omega}^{(\sigma)}
+\dd_y^{-k_1-\sigma} \|u\|_{\alpha_2;\Omega}^{(\sigma)}
\big)^{\frac{\alpha'_2-\alpha'_1}{\alpha'_2}}\,\big(
\|u\|_{\alpha_2;\Omega} |x-y|^{\alpha'_2}
\big)^{\frac{\alpha'_1}{\alpha'_2}}
}{|x-y|^{\alpha'_1} }\\
&\le& C\,\dd^{\alpha_1+\sigma}_{x,y}
\dd_{x,y}^{-\frac{\alpha'_1(\alpha_2+\sigma)}{\alpha'_2}}
\dd^{-\frac{(\alpha'_2-\alpha'_1)(k_1+\sigma)}{\alpha'_2}}_{x,y}\,
\|u\|_{\alpha_2;\Omega}^{(\sigma)}.
\end{eqnarray*}
Moreover, since~$k_1=k_2$,
$$ \alpha_1+\sigma-\frac{\alpha'_1(\alpha_2+\sigma)}{\alpha'_2}
-\frac{(\alpha'_2-\alpha'_1)(k_1+\sigma)}{\alpha'_2}=0,$$
hence the last inequality
proves~\eqref{ma89} when~$k_1=k_2$. Let us
now consider the case~$k_1<k_2$, i.e.~$k_1+1\le k_2$.
Again, we can suppose that~$\dd(x)\le \dd(y)$
and denote by~$B_o$ the ball centered at~$x$
and radius~$\dd(x)/4$, and then~\eqref{ma89.1}
implies that~$y\in B_o$. Also, we
notice that this ball lies in~$\Omega$, at distance
from~$\partial\Omega$
bounded from below by~$3\dd(x)/4$.
As a consequence,
\begin{eqnarray*}
&& |D^{k_1} u(x)-D^{k_1} u(y)|\le
\sup_{ \zeta\in B_o } |D^{k_1+1} u(\zeta)|\,|x-y| \\
&&\qquad\le C \dd_x^{-k_1-1-\sigma}
\sup_{ \zeta\in B_o } \dd^{k_1+1+\sigma}_\zeta |D^{k_1+1} u(\zeta)|\,|x-y|
\\&&\qquad\le C \dd^{-k_1-1-\sigma}_{x,y}\,\|u\|_{\alpha_2;\Omega}^{(\sigma)}
|x-y|.
\end{eqnarray*}
Hence we obtain
\begin{eqnarray*}
&&
\dd^{\alpha_1+\sigma}_{x,y}
\frac{|D^{k_1} u(x)-D^{k_1} u(y)|}{|x-y|^{\alpha'_1} }
\\ &=&
\dd^{\alpha_1+\sigma}_{x,y}
\frac{ |D^{k_1} u(x)-D^{k_1} u(y)|^{1-\alpha'_1}
|D^{k_1} u(x)-D^{k_1} u(y)|^{\alpha'_1}}{|x-y|^{\alpha'_1} }\\
&\le&
\dd^{\alpha_1+\sigma}_{x,y}
\big( \dd^{-k_1-\sigma}_x \|u\|_{\alpha_2;\Omega}^{(\sigma)}+
\dd^{-k_1-\sigma}_y \|u\|_{\alpha_2;\Omega}^{(\sigma)}
\big)^{1-\alpha'_1} \,
\dd^{-\alpha_1'(k_1+1+\sigma)}_{x,y}\, \big(\|u\|_{\alpha_2;\Omega}^{(\sigma)}\big)^{\alpha'_1}
\\ &\le& \dd^{\alpha_1+\sigma}_{x,y} \dd^{-(k_1+\sigma)(1-\alpha'_1)}_{x,y}
\dd^{-\alpha_1'(k_1+1+\sigma)}_{x,y}\,\|u\|_{\alpha_2;\Omega}^{(\sigma)}.
\end{eqnarray*}
Since
$$ \alpha_1+\sigma-(k_1+\sigma)(1-\alpha'_1)-\alpha_1'(k_1+1+\sigma)=0,$$
the inequality above
proves~\eqref{ma89} when~$k_1<k_2$.

Having completed the proof of~\eqref{ma89},
we now notice that
$$ \sum_{j=0}^{k_1} \sup_{x\in\Omega} \left[ \dd^{j+\sigma}(x)\,|D^j u(x)|\right]
\le
\sum_{j=0}^{k_2} \sup_{x\in\Omega} \left[ \dd^{j+\sigma}(x)\,|D^j u(x)|\right],$$
thanks to~\eqref{k}. This inequality and~\eqref{ma89}
imply the desired claim.
\end{proof}

\section{Integral computations for convex sets}\label{S:2a}

The goal of this section is to provide some (somehow optimal)
weighted integral computations of geometric type for convex sets.
Similar estimates for
bounded domains with~$C^{1,1}$ boundary
when the measure~$a$ in \eqref{L} is regular
will be developed in the forthcoming
Section~\ref{S:2b}.
The first geometric integral computation is given by the following:

\begin{lem}\label{AT1}
Let~$p\in \R^n$, $R>2r>0$ and~$\omega\in S^{n-1}$.
Let~$\Omega\subset\R^n$ be a convex open set, with~$B_R(p)\subseteq\Omega$.
Then there exists~$C>0$, possibly\footnote{As a matter of fact, an explicit expression
for such~$C$ is given by~$C=2^{2s+1}$ (which, in particular,
is less than~$8$). This bound can be obtained
at the end of the proof, by using that, if~$f(t):=(1-t)^{-2s}$,
then
$$ \sup_{t\in[0,1/2]} |f'(t)|=
2s \sup_{t\in[0,1/2]} (1-t)^{-2s-1} = (2s)\cdot (1/2)^{-2s-1}.$$ }
depending on~$n$ and~$s$,
such that
\begin{equation}\label{09}
\int_{R}^{+\infty}\,d\rho \;\frac{\chi_\Omega(p+\rho\omega)\,
\chi_{[0,r]} \big( \dr(p+\rho\omega)\big)}{\rho^{1+2s}} \le Cr R^{-1-2s}.\end{equation}
\end{lem}

\begin{proof} The idea is that the set in which the integrand is
non-zero ``typically'' occupies a segment of length comparable to~$r$.

More precisely, consider the half-line $\Theta=\{
p+\rho\omega,\; \rho\ge0\}$.
If~$\Theta$ lies inside $\Omega$, then we are done.
Indeed, in this case, by convexity, the set~$\Omega$
contains the convex envelope of
the ball~$B_R(p)$ and the half-line~$\Theta$, which is an horizontal cylinder of radius~$R$.
In particular, in this case we have that~$\dd(p+\rho\omega)\ge R>r$, hence~$\chi_{[0,r]}
\big( \dr(p+\rho\omega)\big)=0$,
and so the left hand side of~\eqref{09} vanishes.

As a consequence, we may and do assume that~$\Theta\cap (\partial\Omega)$ is non-void.
Thus, we claim that~$\Theta\cap (\partial\Omega)$ only contains a point.
Indeed, suppose by contradiction
that~$p+\rho_1\omega$ and~$p+\rho_2\omega$ belong to~$\partial\Omega$,
for some~$\rho_2>\rho_1\ge0$. In particular, there exists
a sequence~$q_j\in\Omega$ such that~$q_j\to p+\rho_2\omega$
as~$j\to+\infty$.

Then, by convexity, $\Omega$ contains the convex envelope~$K_j$
of~$B_R(p)$ with~$q_j$. Notice that~$K_j$, as $j\to+\infty$,
approaches the convex envelope
of~$B_R(p)$ with~$p+\rho_2\omega$: therefore, for large~$j$,
the point~$p+\rho_1\omega$ belongs to the interior of~$K_j$
and thus to~$\Omega$. This is a contradiction, and so we have shown that
$\Theta\cap(\partial\Omega)$ consists of exactly one point,
that we denote by~$q_\star=p+\rho_\star\omega$, for some~$\rho_\star\ge0$.

We remark that, since~$B_R(p)$ lies in~$\Omega$, we have that
\begin{equation}\label{099}
\rho_\star\ge R.
\end{equation}
Now we show that
\begin{equation}\label{TRIG}
{\mbox{if $\rho\ge0$, $p+\rho\omega\in\Omega$ and $\dr(p+\rho\omega) \in[0,r]$,
then~$\rho\in \left[ \rho_\star(1-rR^{-1}),\, \rho_\star\right]$.}}
\end{equation}
Indeed, by convexity, $\Omega$ contains the interior
of the convex envelope~$K_\star$ of~$B_R(p)$ with~$q_\star$.
Therefore, $\dd(p+\rho\omega)$ is controlled from below by the distance
of~$p+\rho\omega$ to~$\partial K_\star$, which will be denoted
by~$\delta$.

We remark that~$K_\star$ has a radially symmetric
conical singularity at~$q_\star$.
If we let~$\beta$ be the planar angle of such cone, by trigonometry we
have that
\begin{eqnarray*}&& \delta=(\rho_\star-\rho)\,\sin (\beta/2)\\
{\mbox{and }}&& R=\rho_\star\,\sin (\beta/2).\end{eqnarray*}
Thus
\begin{equation}\label{08}
\dd(p+\rho\omega)\ge \delta = \frac{(\rho_\star-\rho)\,R}{\rho_\star}
.\end{equation}
So, if~$\dr(p+\rho\omega)\in[0,r]$ we have that
$$ r\ge \frac{(\rho_\star-\rho)\,R}{\rho_\star},$$
that is~$\rho\ge \rho_\star(1-r R^{-1})$,
which proves~\eqref{TRIG}.

Therefore, using~\eqref{TRIG} and the substitution~$t=\rho_\star^{-1}\rho$, we conclude that
\begin{equation*}\begin{split}
& \int_{R}^{+\infty}\,d\rho \;\frac{\chi_\Omega(p+\rho\omega)\,
\chi_{[0,r]} \big( \dr(p+\rho\omega)\big)}{\rho^{1+2s}}\le
\int_{\rho_\star(1-r R^{-1})}^{\rho_\star}
\frac{d\rho}{\rho^{1+2s}} \\
&\qquad=\frac{1}{\rho_\star^{2s}}
\int_{1-r R^{-1}}^{1}
\frac{dt}{t^{1+2s}}\le \frac{C}{\rho_\star^{2s}}
\left( \frac{1}{(1-r R^{-1})^{2s}}-1\right) \le \frac{Cr R^{-1}}{\rho_\star^{2s}}.
\end{split}\end{equation*}
This and~\eqref{099} imply~\eqref{09}.
\end{proof}

\begin{rem}
We notice that the estimate in~\eqref{09} is optimal, since one
can consider, for instance,
the case in which~$p:=0$ and~$\Omega:=B_{3R}$.
Indeed, in this case, if~$\rho\in [3R-r,3R]$, then~$\rho\omega
\in \partial B_\rho$, thus~$\dr(\rho\omega)= 3R-\rho \in [0,r]$,
and so
\begin{eqnarray*}
&& \int_{R}^{+\infty}\,d\rho \;\frac{\chi_\Omega(p+\rho\omega)\,
\chi_{[0,r]} \big( \dr(p+\rho\omega)\big)}{\rho^{1+2s}} \ge
\int_{3R-r}^{3R} \frac{
d\rho}{\rho^{1+2s}} \\
&&\qquad= \frac{1}{2s} \left( \frac{1}{(3R-r)^{2s}}
-\frac{1}{(3R)^{2s}}\right) =
\frac{1}{2s\,(3R)^{2s}} \left( \frac{1}{(1-\frac{r}{3R})^{2s}}-1\right)\geq \frac{c\,r}{R^{1+2s}},
\end{eqnarray*}
from which it follows
that~\eqref{09} is sharp.
\end{rem}

\begin{rem} The convexity assumption in Lemma \ref{AT1} cannot be dropped.
As a counterexample, let us endow~$\R^n$ with
coordinates~$(x',x_n)\in\R^{n-1}\times\R$, let us take~$p=0$, $\omega=e_1$,
and
$$ \Omega=B_R\cup \big\{|x_n|< R e^{-|x'|/R}\big\}.$$
Notice that~$\Omega$ is not convex.
Also, the points of the form~$(\rho,0,\cdots,0,R e^{-\rho/R})$
belong to~$\partial\Omega$, for any~$\rho\ge 1$. Hence
$$\dr(\rho\omega)\le R e^{-\rho/R}\le r$$
for any~$\rho\ge R\log (Rr^{-1})$.
Hence
$$
\int_{R}^{+\infty}\,d\rho \;\frac{\chi_\Omega(p+\rho\omega)\,
\chi_{[0,r]} \big( \dr(p+\rho\omega)\big)}{\rho^{1+2s}}
\ge
\int_{R\log (Rr^{-1})}^{+\infty}\frac{d\rho}{\rho^{1+2s}}
= \frac{C}{(R\log (Rr^{-1}))^{2s}}.$$
This quantity is larger than the right hand side of~\eqref{09}
when~$r$ is small, and this shows that the convexity assumption is
essential in such result.
\end{rem}

Next is a variation of Lemma~\ref{AT1}. Namely, the integral
on the left hand side
is modified by a weight depending on the distance
(and the condition~$\dr(p+\rho\omega)\in[0,r]$ is dropped).

\begin{lem}\label{AT2}
Let~$\alpha\in[s,1+s)$.
Let~$p$, $q\in \R^n$, $R>0$ and~$\omega\in S^{n-1}$.
Let~$\Omega\subset\R^n$ be a convex open set, with~$B_R(p)\cup B_R(q)
\subseteq\Omega$.
Then there exists~$C>0$, possibly\footnote{More precisely, the constant~$C$ can be
bounded by~$C'/s$, with~$C'$ independent of~$s$.
To see that, one can modify the proof by arguing as follows.
First of all, one notices that~$1+\alpha-s\ge 1$ and that
that
$$ \int_{1/2}^1 \frac{dt}{(1-t)^{\alpha-s} t^{1+2s}}\le
2^{1+2s}\int_{1/2}^1 \frac{dt}{(1-t)^{\alpha-s} }\le  C'',$$
with~$C''$ independent of~$s$.
This says that~$\psi(\mu)\le\psi(1/2)\le C''$ for any~$\mu
\in [1/2,1]$. Conversely, if~$\mu\in(0,1/2)$, then
\begin{eqnarray*}&& \psi(\mu)\le \mu^{2s}
\int_\mu^{1/2} \frac{dt}{(1-t)^{\alpha-s} t^{1+2s}} + C''
\le\mu^{2s}\int_\mu^{1/2} \frac{dt}{(1/2)^{\alpha-s} t^{1+2s}} +C''
\le \frac{C'''}{s}+C'',\end{eqnarray*}
with~$C'''$ independent of~$s$, and therefore
$$ \sup_{\mu\in(0,1]} \psi(\mu)\le \frac{C''''}{s},$$
with~$C''''$ independent of~$s$.}
depending on~$\alpha$, $n$ and~$s$,
such that
\begin{equation}\label{09BIS}
\int_{R}^{+\infty}\,d\rho \;\frac{
\chi_\Omega(p+\rho\omega)\,\chi_\Omega(q+\rho\omega)
}{\dr^{\alpha-s}(p+\rho\omega,q+\rho\omega)\,
\rho^{1+2s}} \le C R^{-s-\alpha}.\end{equation}
\end{lem}

\begin{proof}
We let~$\rho_\star=\sup\{ \rho {\mbox{ s.t. }}
p+\rho\omega\in \Omega\}\in [R,+\infty]$.
By convexity and trigonometry (see e.g.~\eqref{08}), we have that
$$\dd(p+\rho\omega)\ge
\frac{(\rho_\star-\rho)\,R}{\rho_\star},$$
with the obvious limit
notation that the formula above
reads~$\dd(p+\rho\omega)\ge R$ if~$\rho_\star=+\infty$.
Since the same formula holds for~$q$ instead of~$p$, we have that
$$ \dd(p+\rho\omega,q+\rho\omega)=\min\big\{\dd(p+\rho\omega),
\,\dd(q+\rho\omega)\big\}\ge \frac{(\rho_\star-\rho)\,R}{\rho_\star}.$$
Thus, the left hand side of~\eqref{09BIS} is bounded by
\begin{equation}\label{07}
\int_{R}^{\rho_\star}\,d\rho \;
\frac{\rho_\star^{\alpha-s}}{(\rho_\star-\rho)^{\alpha-s}\,R^{\alpha-s}\,\rho^{1+2s}}=
R^{s-\alpha}\rho_\star^{-2s} \int_{\mu_o}^1
\frac{dt}{(1-t)^{\alpha-s} t^{1+2s}},
\end{equation}
where we used the substitution~$t=\rho/\rho_\star$
and the notation~$\mu_o=R/\rho_\star\in(0,1]$.
Now, for any~$\mu\in(0,1]$, we set
$$ \psi(\mu)= \mu^{2s} \int_\mu^1 \frac{dt}{(1-t)^{\alpha-s} t^{1+2s}}$$
and we claim that
\begin{equation}\label{06}
\sup_{\mu\in(0,1]}\psi(\mu)\le C,
\end{equation}
for some~$C>0$. To check this, we recall that~$\alpha-s<1$
and we
notice that~$\psi(1)=0$.
Also,
$$ \int_0^1
\frac{dt}{(1-t)^{\alpha-s} t^{1+2s}}=+\infty,$$
so we compute, by de l'{H}\^opital rule,
$$ \lim_{\mu\to0}\psi(\mu)=
\lim_{\mu\to0}
\frac{ \int_\mu^1
\frac{dt}{(1-t)^{\alpha-s} t^{1+2s}} }{\mu^{-2s}}
= \lim_{\mu\to0}
\frac{ \;\;
\;\;
\frac{1}{(1-\mu)^{\alpha-s} \mu^{1+2s}}\;\;
\;\; }{2s\mu^{-1-2s}}
=\frac{1}{2s}.$$
Thus~$\psi$ can be extended to a continuous
function in~$[0,1]$ and so~\eqref{06} follows.

Then, we can bound~\eqref{07} using~\eqref{06}: we obtain
that the quantity in~\eqref{07} is controlled by
$$ R^{s-\alpha}\rho_\star^{-2s}\mu_o^{-2s}\psi(\mu_o)\le
CR^{s-\alpha}\rho_\star^{-2s}\mu_o^{-2s} =
CR^{s-\alpha}\rho_\star^{-2s} \cdot
R^{-2s}\rho_\star^{2s},$$
which proves~\eqref{09BIS}.
\end{proof}

\begin{rem}
We notice that the estimate in~\eqref{09BIS} is also optimal,
as one can see by taking~$p:=q:=0$ and~$\Omega:=B_{3R}$.
Indeed, in this case, if~$\rho\in [R,2R]$,
then the point~$\rho\omega$ lies on~$\partial B_\rho$
and so it is at distance~$3R-\rho\in[ R,2R]$ from~$\partial\Omega=\partial
B_{3R}$: hence, in this case
\begin{eqnarray*}
&& \int_{R}^{+\infty}\,d\rho \;\frac{
\chi_\Omega(p+\rho\omega)\,\chi_\Omega(q+\rho\omega)
}{\dr^{\alpha-s}(p+\rho\omega,q+\rho\omega)\,
\rho^{1+2s}} \ge
\int_{R}^{2R} \frac{d\rho}{(2R)^{\alpha-s} \rho^{1+2s}}
\\ &&\qquad=\frac{R^{-s-\alpha}}{
2^{\alpha-s+1}s} \left( 1-\frac{1}{2^{2s}}\right)=c\,R^{-s-\alpha},
\end{eqnarray*}
thus showing the optimality~\eqref{09BIS}.
\end{rem}

For further reference, we point out that Lemma~\ref{AT2}
is of course valid in particular when~$p=q$. 
In this case, its
statement simplifies in the following way:

\begin{lem}\label{AT2-simple}
Let~$\alpha\in[s,1+s)$.
Let~$p\in \R^n$, $R>0$ and~$\omega\in S^{n-1}$.
Let~$\Omega\subset\R^n$ be a convex open set, with~$B_R(p)\subseteq\Omega$.
Then there exists~$C>0$, possibly depending on~$\alpha$, $n$ and~$s$,
such that
\begin{equation*}
\int_{R}^{+\infty}\,d\rho \;\frac{
\chi_\Omega(p+\rho\omega)
}{\dr^{\alpha-s}(p+\rho\omega)\,
\rho^{1+2s}} \le C R^{-s-\alpha}.\end{equation*}
\end{lem}

The next integral computation is a simple, but operational, consequence of
elementary geometry:

\begin{lem}\label{dsfvgddseerw}
Let~$R>0$,
$p\in B_R$,
and~$\omega\in S^{n-1}$.
Let~$\Omega\subset\R^n$ be a convex open set, with~$
B_{3R}\cap (\partial\Omega)\ne\varnothing$.
Then there exists~$C>0$, possibly depending on~$n$ and~$s$,
such that
\begin{equation*}
\int_{R}^{+\infty}\,d\rho \;\frac{
\dr^{s}(p+\rho\omega)
}{
\rho^{1+2s}} \le C R^{-s}.\end{equation*}
\end{lem}

\begin{proof} Let~$q_o\in
B_{3R}\cap (\partial\Omega)$. Then
$$ \dr(p+\rho\omega)\le |p+\rho\omega-q_o|\le |p|+|q|+\rho\le C(R+\rho).$$
Therefore, using the substitution~$\rho=Rt$, we obtain
\begin{equation*}
\int_{R}^{+\infty}\,d\rho \;\frac{\dr^{s}(p+\rho\omega)
}{
\rho^{1+2s}} \le C
\int_{R}^{+\infty}\,d\rho \;\frac{
(R+\rho)^s
}{
\rho^{1+2s}} =
CR^{-s}
\int_{1}^{+\infty}\,dt \;\frac{
(1+t)^s
}{
t^{1+2s}}, \end{equation*}
that gives the desired result.
\end{proof}

\section{Integral computations for $C^{1,1}$ domains and bounded measures}\label{S:2b}

In this section, we consider
bounded domains with~$C^{1,1}$ boundary
and bounded measures~$a$
and we obtain the corresponding results
of Section~\ref{S:2a} in this framework. More precisely,
we obtain the results analogous to Lemmata~\ref{AT1},
\ref{AT2}, \ref{AT2-simple}
and~\ref{dsfvgddseerw}, when the convexity assumption on the domain
is replaced by a regularity assumption on the domain and the boundedness of the measure.

The counterpart of Lemma~\ref{AT1} in the
setting of this section is the following:

\begin{lem}\label{AT1:bis}
Let~$p\in \R^n$ and $R>2r>0$.
Let~$\Omega\subset\R^n$ be a
bounded domain with~$C^{1,1}$ boundary, with~$B_R(p)\subseteq\Omega$.
Then there exists~$C>0$, possibly depending on~$n$, $s$ and~$\Omega$,
such that
\begin{equation}\label{7890jj9900} \int_{\R^n\setminus B_R}
\frac{\chi_\Omega(p+x)\,
\chi_{[0,r]} \big( \dr(p+x)\big)}{|x|^{n+2s}}\,dx \le Cr R^{-1-2s}.\end{equation}
\end{lem}

\begin{proof}
First, notice that, by possibly replacing~$\Omega$
with the translated domain~$\Omega-p$,
we may assume that $p=0$.

Also, we can suppose that~$R$ is less then the diameter of~$\Omega$,
otherwise the condition~$B_R\subseteq\Omega$ cannot hold.
Then, we perform the change of variables $z=\frac{x}{R}$, so that~\eqref{7890jj9900}
becomes
\begin{equation}\label{7890jj9900.2}
\int_{\R^n\setminus B_1} \frac{
\chi_{\Omega_R}(z)\,\chi_{[0,\bar r]}
\big( {\rm dist}\,(z,\partial\Omega_R)\big)}{|z|^{n+2s}}\,dz \le C\bar r,\end{equation}
where we denoted $\Omega_R=\frac1R \Omega$ and $\bar r=\frac{r}R$.
Notice that $\Omega_R$ has a bounded $C^{1,1}$ norm (uniformly in $R$), and
in fact converges to a half-space as $R\rightarrow0^+$.
As a consequence, we can apply Proposition~\ref{appa}:
we obtain that there exists~$\kappa_*>0$ such
that, if~$\bar r\in(0,\kappa_*]$ then
\begin{equation}\label{0doxvddg}
\begin{split}
& \big|
\big\{
x\in \Omega_R \cap ( B_{2^{k+1}}\setminus B_{2^k}) {\mbox{ s.t. }}
{\rm dist}\,(x,\partial\Omega_R)\in[0,\bar r]\big\}\big|\\
&\qquad
\le C \bar r
\,{\mathcal{H}}^{n-1} \big( (\partial\Omega_R)
\cap (B_{2^{k+2}}\setminus B_{2^{k-1}})
\big)
\end{split}\end{equation}
for every~$k\in\N$.
Now we estimate the latter term by scaling back
to the original domain and exploiting Lemma~\ref{app0987uv}.
We obtain
\begin{eqnarray*}
&& {\mathcal{H}}^{n-1} \big( (\partial\Omega_R)
\cap (B_{2^{k+2}}\setminus B_{2^{k-1}})\big) \\
&=& \frac{1}{R^{n-1}}
{\mathcal{H}}^{n-1} \big( (\partial\Omega)
\cap (B_{2^{k+2}R}\setminus B_{2^{k-1}R})\big)
\\ &\le& \frac{C (2^{k-1} R)^{n-1}}{R^{n-1}}.
\end{eqnarray*}
This and~\eqref{0doxvddg} give that
$$ \big|
\big\{
x\in \Omega_R \cap ( B_{2^{k+1}}\setminus B_{2^k}) {\mbox{ s.t. }}
{\rm dist}\,(x,\partial\Omega_R)\in[0,\bar r]\big\}\big|\\
\le C 2^{k(n-1)}\bar r .$$
As a consequence, if~$\bar r\in(0,\kappa_*]$,
\begin{eqnarray*}&&
\int_{ B_{2^{k+1}}\setminus B_{2^k} }
\frac{\chi_{\Omega_R}(z)\,\chi_{[0,{\bar r}]}
\big( {\rm dist}\,(z,\partial\Omega_R)\big)}{|z|^{n+2s}}\,dz
\\ &\le&
\int_{ B_{2^{k+1}}\setminus B_{2^k} }
\frac{\chi_{\Omega_R}(z)\,\chi_{[0,{\bar r}]}
\big( {\rm dist}\,(z,\partial\Omega_R)\big)}{2^{k(n+2s)}}\,dz
\\&\le&
\frac{ C \bar r}{ 2^{k(1+2s)}}.
\end{eqnarray*}
By summing over~$k$, we obtain that
\begin{eqnarray*}
&& \int_{\R^n\setminus B_1}
\frac{\chi_{\Omega_R}(z)\,\chi_{[0,{\bar r}]}
\big( {\rm dist}\,(z,\partial\Omega_R)\big)}{|z|^{n+2s}}\,dz \\&\le&
\sum_{k\ge 0}
\int_{ B_{2^{k+1}}\setminus B_{2^k} }
\frac{\chi_{\Omega_R}(z)\,\chi_{[0,{\bar r}]}
\big( {\rm dist}\,(z,\partial\Omega_R)\big)}{|z|^{n+2s}}\,dz
\\ &\le&\sum_{k\ge 0}
\frac{ C \bar r
}{2^{k(1+2s)}}\\&\le& C \bar r,
\end{eqnarray*}
up to renaming~$C$.
This gives~\eqref{7890jj9900.2}
when~$\bar r\le \kappa_*$, with a constant $C$ that does not depend on $s\in(0,1)$.

If instead~$\bar r>\kappa_*$, then
$$\int_{\R^n\setminus B_1} \frac{
\chi_{\Omega_R}(z)\,\chi_{[0,\bar r]}
\big( {\rm dist}\,(z,\partial\Omega_R)\big)}{|z|^{n+2s}}\,dz
\le
\int_{\R^n\setminus B_1}
\frac{dz}{|z|^{n+2s}}\le C\le \frac{C\bar r}{\kappa_*}.$$
This shows that~\eqref{7890jj9900.2}
also holds in this case, up to renaming constants,
and this completes the proof of
Lemma~\ref{AT1:bis}.
Notice that the constant $C$ can be bounded by $C'/s$, with $C'$ independent of $s$.
\end{proof}

Following is the variation of Lemma~\ref{AT2-simple}
needed in the setting of this section:

\begin{lem}\label{AT2-simple:bis}
Let~$\alpha\in[s,1+s)$, $p\in\R^n$ and~$R>0$.
Let~$\Omega\subset\R^n$ be a
bounded domain with~$C^{1,1}$ boundary, with~$B_R(p)
\subseteq\Omega$.
Then there exists~$C>0$, possibly depending on~$\alpha$, $n$, $s$
and~$\Omega$,
such that
\begin{equation}\label{09BIS:bis}
\int_{\R^n\setminus B_R} \frac{
\chi_\Omega(p+x)
}{\dr^{\alpha-s}(p+x)\,
|x|^{n+2s}}\,dx \le C R^{-s-\alpha}.\end{equation}
\end{lem}

\begin{proof} Up to a translation of the domain, we suppose that~$p=0$.
In addition, we can suppose that~$R$ is less then the diameter of~$\Omega$,
otherwise the condition~$B_R\subseteq\Omega$ cannot hold.
Hence we do the change of variables $z=\frac{x}{R}$, so that~\eqref{09BIS:bis}
reduces to
\begin{equation}\label{7890fsdddss-1}
\int_{\R^n\setminus B_1} \frac{
\chi_{\Omega_R}(z)
}{{\rm dist}^{\alpha-s}(z,\partial\Omega_R)\,
|z|^{n+2s}}\,dz \le C
\end{equation}
where we denoted $\Omega_R=\frac1R \Omega$.
Since~$\Omega_R$ has a bounded $C^{1,1}$ norm (uniformly in $R$,
and indeed it converges to a half-space as $R\rightarrow0^+$),
we can apply Corollary~\ref{CR} and obtain that there exists~$\kappa_*>0$
such that, for any~$t\in(0,\kappa_*]$,
\begin{eqnarray*} &&{\mathcal{H}}^{n-1}
\big( \{ z\in \Omega_R\cap(B_{2^{k+1}}\setminus B_{2^k})
{\mbox{ s.t. }}
{\rm dist} (z,\partial\Omega_R)=t
\}\big) \\&&\qquad\le C\,
{\mathcal{H}}^{n-1}
\big( (\partial\Omega_R)\cap(B_{2^{k+2}}\setminus B_{2^{k-1}})
\}\big),\end{eqnarray*}
for any~$k\in\N$.
Furthermore, by Lemma~\ref{app0987uv},
$$
{\mathcal{H}}^{n-1}
\big( (\partial\Omega_R)\cap(B_{2^{k+2}}\setminus B_{2^{k-1}})
\}\big)\le C (2^{k-1})^{n-1}.$$
The latter two formulas give that
$$ {\mathcal{H}}^{n-1}
\big( \{ z\in \Omega_R\cap(B_{2^{k+1}}\setminus B_{2^k})
{\mbox{ s.t. }}
{\rm dist} (z,\partial\Omega_R)=t
\}\big) \le C (2^{k-1})^{n-1}.$$
Consequently, by Coarea Formula,
\begin{equation}\label{0ds4ergerg}
\begin{split}
& \int_{ {\R^n\setminus B_1}\atop{ \{ {\rm dist}\le\kappa_*\} }}
\frac{ \chi_{\Omega_R}(z) }{
{\rm dist}^{\alpha-s}(z,\partial\Omega_R)\,
|z|^{n+2s} }\,dz  \\&\qquad
\le\sum_{k\ge0}
\int_{ {B_{2^{k+1}}\setminus B_{2^k}}\atop{\{ {\rm dist}\le\kappa_*\}}}
\frac{\chi_{\Omega_R}(z)}{{\rm dist}^{\alpha-s}(z,\partial\Omega_R)\,
|z|^{n+2s}}\,dz \\ &\qquad\le
\sum_{k\ge0} \frac{1}{2^{k(n+2s)}}
\int_{{B_{2^{k+1}}\setminus B_{2^k}}\atop{\{ {\rm dist}\le\kappa_*\}}}
\frac{\chi_{\Omega_R}(z)\,|\nabla {\rm dist}(z,\partial\Omega_R)|
}{ {\rm dist}^{\alpha-s}(z,\partial\Omega_R)\,}\,dz
\\ &\qquad
\le\sum_{k\ge0} \frac{1}{2^{k(n+2s)}} \int_0^{\kappa_*}dt\,
\int_{ {B_{2^{k+1}}\setminus B_{2^k}}\atop{
\{ {\rm dist}(z,\partial\Omega_R)=t\}}}
d{\mathcal{H}}^{n-1}(z)\,
\frac{1}{t^{\alpha-s}}\\
&\qquad\le
\sum_{k\ge0} \int_0^{\kappa_*}dt\,\frac{
{\mathcal{H}}^{n-1}
\big( \{ z\in \Omega_R\cap(B_{2^{k+1}}\setminus B_{2^k})
{\mbox{ s.t. }}
{\rm dist} (z,\partial\Omega_R)=t
\}\big)
}{t^{\alpha-s}\,2^{k(n+2s)}}\\ &\qquad
\le
\sum_{k\ge0} \int_0^{\kappa_*}dt\,
\frac{C (2^{k-1})^{n-1}}{t^{\alpha-s}\,2^{k(n+2s)}}
\\&\qquad\le
\sum_{k\ge0} \int_0^{\kappa_*}dt\,
\frac{C}{t^{\alpha-s}\,2^{k(1+2s)}}
\\ &\qquad=\sum_{k\ge0}
\frac{C\kappa_*^{1-\alpha+s}}{2^{k(1+2s)}} \\&\qquad \le C,
\end{split}\end{equation}
up to renaming constants.
Additionally, we have that
$$ \int_{ {\R^n\setminus B_1}\atop{ \{ {\rm dist}\ge\kappa_*\} } }
\frac{\chi_{\Omega_R}(z)
}{{\rm dist}^{\alpha-s}(z,\partial\Omega_R)\,
|z|^{n+2s}}\,dz \le
\int_{\R^n\setminus B_1}
\frac{ \chi_{\Omega_R}(z) }{\kappa_*^{\alpha-s}
\,|z|^{n+2s}}\,dz
\le C.$$
This and~\eqref{0ds4ergerg}
complete the proof of~\eqref{7890fsdddss-1}
and thus of Lemma~\ref{AT2-simple:bis}.
The constant $C$ can be bounded by $C'/s$, with $C'$ independent of $s$.
\end{proof}

Then, the result corresponding to Lemma~\ref{AT2} goes as follows:

\begin{lem}\label{AT2:bis}
Let~$\alpha\in[s,1+s)$.
Let~$p$, $q\in \R^n$ and~$R>0$.
Let~$\Omega\subset\R^n$ be a
bounded domain with~$C^{1,1}$ boundary, with~$B_R(p)\cup B_R(q)
\subseteq\Omega$.
Then there exists~$C>0$, possibly depending on~$\alpha$, $n$, $s$
and~$\Omega$,
such that
\begin{equation*}
\int_{\R^n\setminus B_R} \frac{
\chi_\Omega(p+x)\,\chi_\Omega(q+x)
}{\dr^{\alpha-s}(p+x,q+x)\,
|x|^{n+2s}}\,dx \le C R^{-s-\alpha}.\end{equation*}
\end{lem}

\begin{proof} Notice that $\dr(p+x,q+x)$ coincides
with either~$\dd(p+x)$ or~$\dd(q+x)$, therefore
$$ \frac{1}{\dr^{\alpha-s}(p+x,q+x)}
\le \frac{1}{\dr^{\alpha-s}(p+x)}+\frac{1}{\dr^{\alpha-s}(q+x)}.$$
Furthermore,~$\chi_\Omega(p+x)\,\chi_\Omega(q+x)\le
\chi_\Omega(p+x)$ and~$\chi_\Omega(p+x)\,\chi_\Omega(q+x)\le
\chi_\Omega(q+x)$. As a consequence
$$
\frac{
\chi_\Omega(p+x)\,\chi_\Omega(q+x)
}{\dr^{\alpha-s}(p+x,q+x)\,
|x|^{n+2s}}\le
\frac{
\chi_\Omega(p+x)
}{\dr^{\alpha-s}(p+x)\,
|x|^{n+2s}}+
\frac{
\chi_\Omega(q+x)
}{\dr^{\alpha-s}(q+x)\,
|x|^{n+2s}},$$
and so Lemma~\ref{AT2:bis} follows from
Lemma~\ref{AT2-simple:bis}.
In particular, the constant $C$ can be bounded by $C'/s$, with $C'$ independent of $s$.
\end{proof}

Finally, here is the counterpart of Lemma~\ref{dsfvgddseerw}:

\begin{lem}\label{dsfvgddseerw:bis}
Let~$R>0$ and~$p\in B_R$.
Let~$\Omega\subset\R^n$ be a
bounded domain with~$C^{1,1}$ boundary,
with~$
B_{3R}\cap (\partial\Omega)\ne\varnothing$.
Then there exists~$C>0$, possibly depending on~$n$, $s$ and~$\Omega$,
such that
\begin{equation*}
\int_{\R^n\setminus B_R}
\frac{
\dr^{s}(p+x)
}{|x|^{n+2s}}\,dx \le C R^{-s}.\end{equation*}
\end{lem}

\begin{proof} Let~$q_o\in B_{3R}\cap (\partial\Omega)$. Then
$$ \dd(p+x)\le |p+x-q_o|\le |p|+|x|+|q_o|\le |x|+4R,$$
therefore
\begin{equation*}
\int_{\R^n\setminus B_R}
\frac{
\dr^{s}(p+x)
}{|x|^{n+2s}}\,dx \le
\int_{\R^n\setminus B_R}
\frac{
(|x|+4R)^s
}{|x|^{n+2s}}\,dx = R^{-s}
\int_{\R^n\setminus B_1}
\frac{
(|y|+4)^s
}{|y|^{n+2s}}\,dy=CR^{-s}.
\end{equation*}
Notice that the constant $C$ can be bounded by $C'/s$, with $C'$ independent of $s$.
\end{proof}

\section{A localization argument}\label{US}

In this section we introduce an appropriate cutoff function
and we discuss its regularity properties. The goal of the cutoff
procedure is, roughly speaking, to distinguish the behavior
of the solutions inside the domain from the one at the boundary.
For this we recall the notation in~\eqref{dxy} and we give the following result:

\begin{lem}\label{WW} Let~$R>0$, $\Omega\subset\R^n$ and
$\alpha\in(0,1+s)$.
Assume that either~\eqref{HYP1} or~\eqref{HYP2} is satisfied
and that
\begin{equation}\label{78fg3992}
{\mbox{$B_{2R}\subseteq\Omega$ and
$B_{3R}\cap (\partial\Omega)\ne\varnothing$.}}\end{equation}
Let~$w\in C^s(\R^n)$,
with
\begin{eqnarray}
&& \label{ze} {\mbox{$w\equiv0$ in~$B_R$}}\\
{\mbox{and }}&& \label{ze2} {\mbox{$w\equiv0$ outside~$\Omega$.}}
\end{eqnarray}
Then
\begin{equation}\label{W1}
\|Lw\|_{L^\infty(B_{R/2})} \le C\,
\,[w]_{C^s(\R^n)} R^{-s}.
\end{equation}
In addition, if we assume also that
\begin{itemize}
\item either~$\alpha\in(0,s]$,
\item or~$\alpha\in(s,1+s)$ and
\begin{itemize}
\item for any~$p$, $q\in\Omega$ with~$|p-q|\le R$ we have that
\begin{equation}\label{L01L}
\left\{\begin{matrix}
|w(p)-w(q)|\le \ell\,|p-q|^\alpha \, \big( \dd^{s-\alpha}(p,q)
+R^{-\alpha}\dd^{s}(p,q)\big) & {\mbox{ if }} \alpha\in(s,1],\\
&\\
\begin{matrix}
|\nabla w(p)|\le\ell\,
\big( \dd^{s-1}(p)
+R^{-1}\dd^{s}(p)\big) \qquad{\mbox{and }}
\\
|\nabla w(p)-\nabla w(q)|\le \ell\,|p-q|^{\alpha-1} \, \big( \dd^{s-\alpha}(p,q)
+R^{-\alpha}\dd^{s}(p,q)\big) \end{matrix}
& {\mbox{ if }} \alpha\in(1,1+s),
\end{matrix}\right.\end{equation}
for some~$\ell>0$,\end{itemize}\end{itemize}
then
\begin{equation}\label{wanted}
R^{\alpha+s}[Lw]_{C^\alpha(B_{R/4})}\leq
\left\{\begin{matrix}
C\,[w]_{C^s(\R^n)}
& {\mbox{ if }} \alpha\in(0,s],
\\
C\,\Big(
[w]_{C^s(\R^n)}+
\ell \Big) & {\mbox{ if }} \alpha\in(s,1+s).
\end{matrix}
\right.
\end{equation}
Finally, if~$\alpha\in(1,1+s)$, we also have that
\begin{equation}\label{Xnuova}
R^{1+s}\|\nabla Lw\|_{L^\infty(B_{R/4})}\le C\ell.
\end{equation}
\end{lem}

\begin{proof}
For simplicity,
we state and prove
this result for convex open sets, i.e. when~\eqref{HYP1}
is assumed.
The proof under condition~\eqref{HYP2} would be the same,
except that one should use the results of
Section~\ref{S:2b} instead of the ones of Section~\ref{S:2a}.
More explicitly, for convex open sets, in the proof of this result
we will quote
Lemmata~\ref{AT1}, \ref{AT2-simple}, \ref{dsfvgddseerw}
and~\ref{AT2}: for bounded domains with~$C^{1,1}$
boundary one has instead to quote
Lemmata~\ref{AT1:bis}, \ref{AT2-simple:bis},
\ref{dsfvgddseerw:bis} and~\ref{AT2:bis}.

First of all, we prove~\eqref{W1}. Fix~$x\in B_{R/2}$.
Then~$w(x)=0$ and~$w(x+\rho \omega)=0$ for any~$\rho\in(-R/2,R/2)$,
thanks to~\eqref{ze}. Accordingly, for all $\rho>0$ we have
$$ |w(x+\rho \omega)|=|w(x+\rho \omega)-w(x)|\le [w]_{C^s(\R^n)} \rho^s$$
therefore
$$ Lw(x)\le 2\,[w]_{C^s(\R^n)}
\int_{R/2}^{+\infty}\,d\rho
\int_{S^{n-1}}\,da(\omega)\,
\frac{1}{\rho^{1+s}}\le C\,[w]_{C^s(\R^n)} R^{-s},$$
which proves~\eqref{W1}.

Now we prove~\eqref{wanted}.
For this, we first consider the case~$\alpha\in(0,1]$,
and we fix~$x_1$ and $x_2$ in $B_{R/4}$.
Notice that if~$y\in B_{R/2}$ then~$w(x_1+y)=w(x_2+y)=0$, thanks to~\eqref{ze}.
In particular, we have $w(x_1)=w(x_2)=0$.
As a consequence of these observations,
\begin{equation*}
\begin{split}
Lw(x_i)\,&=-\int_{R/2}^{+\infty}\,d\rho
\int_{S^{n-1}}\,da(\omega)\,
\frac{
w(x_i+\rho\omega)+w(x_i-\rho\omega)}{\rho^{1+2s}}\\&=
-2\int_{R/2}^{+\infty}\,d\rho
\int_{S^{n-1}}\,da(\omega)\,
\frac{
w(x_i+\rho\omega)}{\rho^{1+2s}},
\end{split}\end{equation*}
for $i\in\{1,2\}$ (and possibly replacing~$da(\omega)$
with~$da(\omega)+da(-\omega)$). Therefore
\begin{equation}\label{general}
|Lw(x_1)-Lw(x_2)|\leq
2\int_{R/2}^{+\infty}\,d\rho
\int_{S^{n-1}}\,da(\omega)\,
\frac{
|w(x_1+\rho\omega) - w(x_2+\rho\omega)|}{\rho^{1+2s}}
.\end{equation}
So, we distinguish two cases.
If~$\alpha\in(0,s]$, then we obtain from~\eqref{general} that
\[\begin{split}
|Lw(x_1)-Lw(x_2)|& \leq 2 [w]_{C^s(\R^n)} |x_1-x_2|^s
\int_{R/2}^{+\infty}\,d\rho
\int_{S^{n-1}}\,da(\omega)\,
\frac{1}{\rho^{1+2s}}\\
&\leq C [w]_{C^s(\R^n)} |x_1-x_2|^s R^{-2s}.
\end{split}\]
Therefore
$$\frac{ |Lw(x_1)-Lw(x_2)|}{|x_1-x_2|^\alpha}
\le C [w]_{C^s(\R^n)} |x_1-x_2|^{s-\alpha} R^{-2s}.$$
So, if~$\alpha\in(0,s]$, the result in~\eqref{wanted}
follows by noticing that~$|x_1-x_2|\le |x_1|+|x_2|\le R$.

Now suppose that~$\alpha\in(s,1]$.
We define~$\dd_\star(\rho)=\dr(x_1+\rho\omega)+\dr(x_2+\rho\omega)$
and we write~\eqref{general} as
\begin{equation}\label{LX}
|Lw(x_1)-Lw(x_2)| \leq
 I_1+I_2,\end{equation}
where $r=|x_1-x_2|$,
\begin{eqnarray*}
&& I_1=2\int_{R/2}^{+\infty}\,d\rho
\int_{S^{n-1}}\,da(\omega)\,\chi_{[0,6r]} \big(\dd_\star(\rho)\big)\,
\frac{
|w(x_1+\rho\omega) - w(x_2+\rho\omega)|}{\rho^{1+2s}}\\
{\mbox{and }}&& I_2=
2\int_{R/2}^{+\infty}\,d\rho
\int_{S^{n-1}}\,da(\omega)\,\chi_{(6r,+\infty)} \big(\dd_\star(\rho)
\big)\,
\frac{
|w(x_1+\rho\omega) - w(x_2+\rho\omega)|}{\rho^{1+2s}}.
\end{eqnarray*}
Now we estimate~$I_1$. For this, we fix~$\rho\ge0$
such that~$\dd_\star(\rho)\in[0,6r]$.
Thus,
for any~$i\in\{1,2\}$, we have that~$\dr(x_i+\rho\omega)
\le \dd_\star(\rho)\le 6r$, thus,
by~\eqref{ze2},
\begin{eqnarray*}
|w(x_i+\rho\omega)|&=&
|w(x_i+\rho\omega)|\,\chi_\Omega(x_i+\rho\omega) \\
&\le& [w]_{C^s(\R^n)}\,\dr^s(x_i+\rho\omega)\,\chi_\Omega(x_i+\rho\omega)
\\ &\le & C\, [w]_{C^s(\R^n)}\,r^s \,\chi_\Omega(x_i+\rho\omega).\end{eqnarray*}
As a consequence
\begin{equation}\label{estimate I1}
I_1\le C \,[w]_{C^s(\R^n)}\,r^s\,\sum_{i=1}^2
\int_{R/2}^{+\infty}\,d\rho
\int_{S^{n-1}}\,da(\omega)\, \frac{
\chi_\Omega(x_i+\rho\omega)\chi_{[0,6r]} \big( \dr(x_i+\rho\omega)\big)
}{\rho^{1+2s}}
.\end{equation}
This and Lemma~\ref{AT1} give that
\begin{equation}\label{LY}
I_1\le C \,[w]_{C^s(\R^n)}\,r^{1+s} R^{-1-2s},
\end{equation}
for some~$C>0$.

Now we estimate~$I_2$.
For this we let~$\dd_\star(\rho)>6r$ and
we will show that
\begin{equation}\label{alt}
\begin{split} &|w(x_1+\rho\omega) - w(x_2+\rho\omega)|\\&\qquad\le
Cr^\alpha
\ell \big( \dr^{s-\alpha}(x_1+\rho\omega)+R^{-\alpha}\dr^{s}(x_1+\rho\omega)
\big).\end{split}
\end{equation}
To prove this, we observe that
$$ \dr(x_1+\rho\omega)\le \dr(x_2+\rho\omega)+|x_1-x_2|
=\dr(x_2+\rho\omega)+r.$$
Therefore
$$ \dd_\star(\rho)=\dr(x_1+\rho\omega)+\dr(x_2+\rho\omega)\le
2\dr(x_2+\rho\omega)+r.$$
Thus, if~$\dd_\star(\rho)> 6r$ we have that
$$ \frac{5}{12} \dd(x_1+\rho\omega)\le
\frac{5}{12} \dd_\star(\rho)<
\frac{\dd_\star(\rho)-r}{2}\le \dr(x_2+\rho\omega).$$
In particular
$$ \dd(x_1+\rho\omega,\,x_2+\rho\omega)=\min\{ \dd(x_1+\rho\omega),\,\dd(x_2+\rho\omega)\}\ge
\frac{5}{12} \dd(x_1+\rho\omega).$$
Also, of course $\dd(x_1+\rho\omega,\,x_2+\rho\omega)\le \dd(x_1+\rho\omega)$.
As a consequence of these observations,
we can exploit~\eqref{L01L} with~$p=x_1+\rho\omega$
and~$q=x_2+\rho\omega$,
and we obtain
$$ |w(x_1+\rho\omega) - w(x_2+\rho\omega)|\le
\ell\,|x_1-x_2|^\alpha \left(
\left(\frac{5}{12} \dd(x_1+\rho\omega)\right)^{s-\alpha}+R^{-\alpha}\dd^s(x_1+\rho\omega)\right),$$
which implies~\eqref{alt}.

Having completed the proof of~\eqref{alt}, we can use such formula
to obtain that
\begin{eqnarray*} I_2 &\le& C\ell r^\alpha
\int_{R/2}^{+\infty}\,d\rho
\int_{S^{n-1}}\,da(\omega)\,\chi_{(r,+\infty)} \big(\dr(x_1+\rho\omega)
\big)\,
\frac{\dr^{s-\alpha}(x_1+\rho\omega)}{\rho^{1+2s}}\\
&&\qquad+C\ell r^\alpha R^{-\alpha}
\int_{R/2}^{+\infty}\,d\rho
\int_{S^{n-1}}\,da(\omega)\,\chi_{(r,+\infty)} \big(\dr(x_1+\rho\omega)
\big)\,
\frac{\dr^{s}(x_1+\rho\omega)}{\rho^{1+2s}}.\end{eqnarray*}
So we use Lemma~\ref{AT2-simple}
(resp., Lemma~\ref{dsfvgddseerw})
to bound the first (resp., the second) of
the two integrals above: we obtain
\begin{equation}\label{alt-fin}
I_2\le C\ell r^\alpha R^{-s-\alpha} .\end{equation}
This, \eqref{LX} and~\eqref{LY} give that
$$ |Lw(x_1)-Lw(x_2)|\le Cr^\alpha \big(
[w]_{C^s(\R^n)}\,r^{1+s-\alpha} R^{-1-2s}+
\ell R^{-s-\alpha} \big).$$
So, since $1+s-\alpha>0$,
we have that~$r^{1+s-\alpha}\le R^{1+s-\alpha}$ and so
$$ |Lw(x_1)-Lw(x_2)|\le C|x_1-x_2|^\alpha \big(
[w]_{C^s(\R^n)}\,R^{-s-\alpha}+
\ell R^{-s-\alpha}\big),$$
which establishes~\eqref{wanted} when~$\alpha\in(s,1]$.

It remains now to consider the case in which~$\alpha\in(1,1+s)$. For this scope,
we modify the argument above by looking at~$L\partial_j w(x_1)-L\partial_j w(x_2)$,
for a fixed~$j\in\{1,\cdots,n\}$.
In this case, formula~\eqref{general} becomes
\begin{eqnarray*} |L \partial_j w(x_1)-L\partial_j w(x_2)|&\leq&
2\int_{R/2}^{+\infty}\,d\rho
\int_{S^{n-1}}\,da(\omega)\,
\frac{
|\partial_j w(x_1+\rho\omega) - \partial_j w(x_2+\rho\omega)|}{\rho^{1+2s}}
\\&\le& J_1+J_2,\end{eqnarray*}
where $r=|x_1-x_2|$,
\begin{eqnarray*}
&& J_1=2\int_{R/2}^{+\infty}\,d\rho
\int_{S^{n-1}}\,da(\omega)\,\chi_{[0,6r]} \big(\dd_\star(\rho)\big)\,
\frac{
|\partial_jw(x_1+\rho\omega) - \partial_jw(x_2+\rho\omega)|}{\rho^{1+2s}}\\
{\mbox{and }}&& J_2=
2\int_{R/2}^{+\infty}\,d\rho
\int_{S^{n-1}}\,da(\omega)\,\chi_{(6r,+\infty)} \big(\dd_\star(\rho)
\big)\,
\frac{
|\partial_jw(x_1+\rho\omega) - \partial_jw(x_2+\rho\omega)|}{\rho^{1+2s}}.
\end{eqnarray*}
First we estimate~$J_1$.
We fix~$\rho\ge0$
such that~$\dd_\star(\rho)\in[0,6r]$ and,
for any~$i\in\{1,2\}$, we obtain that~$\dr(x_i+\rho\omega)
\le \dd_\star(\rho)\le 6r$. Hence, when both~$x_1+\rho\omega$ and~$x_2+\rho\omega$
belong to~$\Omega$
we deduce from~\eqref{L01L} that
\begin{eqnarray*}
&&|\partial_i w(x_1+\rho\omega)-\partial_i w(x_2+\rho\omega)| \\&
\le& \ell\,|x_1-x_2|^{\alpha-1} \, \big( \dd^{s-\alpha}(x_1+\rho\omega,x_2+\rho\omega)
+R^{-\alpha}\dd^{s}(x_1+\rho\omega,x_2+\rho\omega)\big) \\
&\le& C\ell\,r^{\alpha-1} \,\dd^{s-\alpha}(x_1+\rho\omega,x_2+\rho\omega).
\end{eqnarray*}
This estimate and Lemma~\ref{AT2} imply that
\begin{equation}\label{IS1}
\begin{split}
&\int_{R/2}^{+\infty}\,d\rho
\int_{S^{n-1}}\,da(\omega)\,\chi_{[0,6r]} \big(\dd_\star(\rho)\big)\,
\frac{\chi_\Omega(x_1+\rho\omega)\,\chi_\Omega(x_2+\rho\omega)\,
|\partial_jw(x_1+\rho\omega) - \partial_jw(x_2+\rho\omega)|}{\rho^{1+2s}}
\\ &\quad\le C\ell\,r^{\alpha-1} R^{-s-\alpha}.
\end{split}\end{equation}
If instead~$x_1+\rho\omega\in\Omega$ and~$x_2+\rho\omega\not\in\Omega$,
up to a set of measure zero we have that~$\partial_jw(x_2+\rho\omega)=0$
and so, by~\eqref{L01L},
\begin{equation}\label{SI}\begin{split}
&|\partial_i w(x_1+\rho\omega)-\partial_i w(x_2+\rho\omega)|=
|\partial_i w(x_1+\rho\omega)|\\
&\qquad\le
\ell\,
\big( \dd^{s-1}(x_1+\rho\omega)
+R^{-1}\dd^{s}(x_1+\rho\omega)\big)\le
C\ell\,\dd^{s-1}(x_1+\rho\omega)\\
&\qquad\le
C\ell\,r^{\alpha-1}\,\dd^{s-\alpha}(x_1+\rho\omega)
\end{split}\end{equation}
Notice that in the last two steps we have used the fact that~$\dd(x_1+\rho\omega)
\le \dd_\star(\rho)\le 6r\le 6R$ (together with~$\alpha\ge1$).
Formula~\eqref{SI} and Lemma~\ref{AT2-simple}
imply that
\begin{equation}\label{IS2}
\begin{split}
&\int_{R/2}^{+\infty}\,d\rho
\int_{S^{n-1}}\,da(\omega)\,\chi_{[0,6r]} \big(\dd_\star(\rho)\big)\,
\frac{\chi_\Omega(x_1+\rho\omega)\,\chi_{\R^n\setminus\Omega}(x_2+\rho\omega)\,
|\partial_jw(x_1+\rho\omega) - \partial_jw(x_2+\rho\omega)|}{\rho^{1+2s}}
\\ &\quad\le C\ell\,r^{\alpha-1} R^{-s-\alpha}.
\end{split}\end{equation}
A similar estimate also holds by exchanging~$x_1$ and~$x_2$.
Then,
since also~$|\partial_i w(x_1+\rho\omega)-\partial_i w(x_2+\rho\omega)|=0$
if both~$x_1+\rho\omega$ and~$x_2+\rho\omega$
lie outside~$\Omega$, up to sets of null measure,
we obtain from~\eqref{IS1} and~\eqref{IS2} that
$$ J_1 \le C \ell\,r^{\alpha-1} R^{-s-\alpha}.$$
Now we also bound~$J_2$ by~$C\ell r^{\alpha-1} R^{-s-\alpha}$.
This can be obtained by repeating
the argument from~\eqref{alt} to~\eqref{alt-fin},
replacing~$w$ by~$\partial_iw$ (and~$|x_1-x_2|^\alpha$ by~$|x_1-x_2|^{\alpha-1}$).
The estimates obtained on~$J_1$ and~$J_2$
prove that
\begin{equation*}
|L \partial_j w(x_1)-L\partial_j w(x_2)|\le C\ell\,r^{\alpha-1} R^{-s-\alpha}
,\end{equation*}
up to renaming~$C$, and so
\begin{equation}\label{L01L67890ojh1}
R^{\alpha+s}[L\partial_j w]_{C^{\alpha-1}(B_{R/4})}\leq C\ell.
\end{equation}
Now we observe that
\begin{equation}\label{L01L67890ojh2}
L\partial_j  w=\partial_j Lw.
\end{equation}
This can be proved by using~\eqref{L01L} to obtain that, for any~$x\in\Omega$
and~$h\in\R$ sufficiently small,
\begin{eqnarray*}
&& \left|\frac{w(x+\rho\omega+he_i)-w(x+\rho\omega)}{h} -\partial_i w(x+\rho\omega)\right|\\
&=& \left| \int_0^1 \partial_iw (x+\rho\omega+hte_i)-\partial_iw (x+\rho\omega)\,dt\right|\\
&\le& \int_0^1 \,dt
\ell\,|h|^{\alpha-1} \, \big( \dd^{s-\alpha}(x+\rho\omega+hte_i,x+\rho\omega)
+R^{-\alpha}\dd^{s}(x+\rho\omega+hte_i,x+\rho\omega)\big),
\end{eqnarray*}
and then integrating and using the same argument as above.

{F}rom~\eqref{L01L67890ojh1}
and~\eqref{L01L67890ojh2}, one
completes the proof of~\eqref{wanted} when~$\alpha\in(1,1+s)$.\medskip

It only remains to prove~\eqref{Xnuova}.
For this, we take~$\alpha\in(1,1+s)$ and~$x\in B_{R/4}$.
We remark that~$\dd(x)\in [R,4R]$ and~$w(x+\rho\omega)$ vanishes when~$\rho\in(0,R/2)$
(thus so does~$\partial_j w$, for any~$j\in\{1,\cdots,n\}$),
hence, recalling first~\eqref{L01L67890ojh2}
and then~\eqref{L01L}, we have
\begin{eqnarray*}
&& |\partial_j Lw(x)|\,=\,|L\partial_j w(x)|\\
&\le& 2 \int_{\{ |\rho|>R/2\}} \,d\rho
\int_{S^{n-1}}\,da(\omega) \frac{
|\partial_jw(x+\rho\omega)|}{\rho^{1+2s}}
\\ &\le&
C\ell\int_{\{ |\rho|>R/2\} } \,d\rho
\int_{S^{n-1}}\,da(\omega) \frac{
\chi_\Omega(x+\rho\omega)\,\big(
\dd^{s-1}(x+\rho\omega)+R^{-1} \dd^s(x+\rho\omega)\big)
}{\rho^{1+2s}}.\end{eqnarray*}
The term with $\chi_\Omega(x+\rho\omega)\,
\dd^{s-1}(x+\rho\omega)$
at the numerator can be estimated with $C\ell R^{s-1}$
by means of Lemma~\ref{AT2-simple}
(used here with~$\alpha=1$).
The term with~$\dd^s(x+\rho\omega)$ at the numerator
can also be bounded in this way, using Lemma~\ref{dsfvgddseerw}. {F}rom these
considerations we obtain that~$|\partial_j Lw(x)|
\le C\ell R^{-s-1}$, which establishes~\eqref{Xnuova}.
\end{proof}

\section{Iterative $C^{\alpha+2s}$-regularity}\label{IT:SEC}

The cornerstone of our regularity theory is the following
Theorem~\ref{CS}. Namely, we show that if the solution
lies in some H\"older space, than it indeed lies in
a better H\"older space (with estimates).

\begin{thm}\label{CS}
Let~$\alpha\in(0,1+s)$.
Let~$\Omega\subset\R^n$
and assume that either~\eqref{HYP1} or~\eqref{HYP2} is satisfied.

Let $u\in C^s(\R^n)\cap C^\alpha_{\rm loc}(\Omega)$ be a solution to \eqref{eq},
with~$g\in C^\alpha_{\rm loc}(\Omega)$.

If~$\alpha\in(s,1+s)$ assume in addition that
\begin{equation} \label{12sdv56}
\|u\|_{\alpha; \Omega}^{(-s)} <+\infty
.\end{equation}
Then~$u\in C^{\alpha+2s}_{\rm loc}(\Omega)$ and
\begin{equation}\label{TP-fin}
\|u\|_{\alpha+2s; \Omega}^{(-s)} \le C\,\big( \|g\|_{\alpha; \Omega}^{(s)} +
\|u\|_{C^s(\R^n)} +
\chi_{(s,1+s)}(\alpha)\, \|u\|_{\alpha; \Omega}^{(-s)} \big)
\end{equation}
whenever $\alpha+2s$ is not an integer.
\end{thm}

\begin{proof}
First, notice that \eqref{TP-fin} implies that $u\in C^{\alpha+2s}_{\rm loc}(\Omega)$.
To prove~\eqref{TP-fin}, we fix~$p,q\in\Omega$, $p\ne q$, and we aim to show that
\begin{equation}\label{TP-fin-0}
\begin{split}
& \sum_{j=0}^{k_s} \left( \dd^{j-s}(p)\,|D^j u(p)|
+\dd^{j-s}(q)\,|D^j u(q)|\right)+
\dd^{\alpha+s}(p,q)
\frac{|D^{k_s}u(p)-D^{k_s}u(q)|}{|p-q|^{\alpha'_s}}\\&\qquad\le
C\,\big( \|g\|_{\alpha; \Omega}^{(s)} +
\|u\|_{C^s(\R^n)} +
\chi_{(s,1+s)}(\alpha)\, \|u\|_{\alpha; \Omega}^{(-s)} \big),
\end{split}\end{equation}
where~$k_s\in\N$ and~$\alpha_s'\in(0,1]$ are such that~$\alpha+2s=k_s+\alpha'_s$.
To prove this, we distinguish two cases:
either~$|p-q|<\dd(p,q)/30$ or~$|p-q|\ge \dd(p,q)/30$.

We start with the case~$|p-q|<\dd(p,q)/30$.
Without loss of generality, by possibly exchanging~$p$ and~$q$,
we suppose that
\begin{equation}\label{set1}
\dd(p)\le \dd(q)\end{equation}
and we set
\begin{equation}\label{set2}
R=\frac{\dd(p)}3.\end{equation}
Notice that there exists~$p_\star\in (\partial\Omega)\cap(\partial B_{3R}(p))$
and
$$ |p-q|<\frac{\dd(p,q)}{30}=\frac{\dd(p)}{30}=\frac{R}{10}.$$
Up to a translation, we also suppose that~$p=0$, hence
\begin{eqnarray}
\label{q} && q\in B_{R/10}(p)=B_{R/10},\\
\label{5.F} && \Omega\supseteq B_{3R}(p)=B_{3R}\\
{\mbox{and }} \label{5.FF} &&
p_\star\in (\partial\Omega)\cap (\partial B_{3R}),
\end{eqnarray}
hence formula~\eqref{78fg3992} holds true with this setting.
We let~$\eta_\star\in C^\infty(\R^n,[0,1])$
such that $\eta_\star\equiv1$ in $B_{1}$ and $\eta_\star\equiv0$
outside~$B_{3/2}$. Let us also define~$\eta(x)=\eta_\star(x/R)$.

Let us consider $\bar u=\eta u$, $\vartheta=1-\eta$ and~$w=\vartheta u$.
Since~$\eta_\star$ is fixed once and for all, we can write, for any~$\alpha'\in(0,1]$,
\begin{equation}\label{dcuxj} \begin{split} &
|\vartheta(x)-\vartheta(y)|=|\eta(y)-\eta(x)|=\left| \eta_\star\left(
\frac{y}R\right) - \eta_\star\left( \frac{x}R\right)\right|\\ &\qquad\le
[\eta_\star]_{C^{\alpha'}(\R^n)} \left|\frac{y}R-\frac{x}R\right|^{\alpha'}\le C
R^{-\alpha'} |x-y|^{\alpha'}. \end{split}\end{equation}
Similarly
\begin{equation}\label{dcuxj2}
\nabla\vartheta(x)=-\nabla\eta(x)=-R^{-1}
\nabla\eta_\star\left( \frac{x}R\right)
\end{equation}
and so, for any~$\alpha'\in(0,1]$,
\begin{equation}\label{dcuxj3}
\begin{split}
&|\nabla\vartheta(x)-\nabla\vartheta(y)|=|\nabla\eta(x)-\nabla\eta(y)|
\\&\qquad\le
R^{-1-\alpha'}\|\eta_\star\|_{C^{1+\alpha'}}|x-y|^{\alpha'}\le
CR^{-1-\alpha'}|x-y|^{\alpha'}.
\end{split}\end{equation}
Notice also
that~$w\equiv 0$ in~$B_{R}$ and~$w\equiv0$ outside~$\Omega$. Our goal is
to show that $w$ satisfies the assumptions of Lemma~\ref{WW}. For this,
when~$\alpha\in(s,1+s)$, we need to check condition~\eqref{L01L}.
To this goal, we claim that,
if~$\alpha\in(s,1+s)$, then
\begin{equation}\label{WA} \begin{split}
&{\mbox{condition~\eqref{L01L} holds true with}}\\
&\qquad
\ell=
C\,\big( \|u\|_{\alpha; \Omega}^{(-s)} +[u]_{C^s(\R^n)}\big).
\end{split}\end{equation}
To prove~\eqref{WA}, we split the cases~$\alpha\in(s,1]$
and~$\alpha\in(1,1+s)$.

Let us first deal with the case
\begin{equation}\label{cc}
\alpha\in(s,1].
\end{equation}
We fix~$x$, $y\in\Omega$ with~$|x-y|\le R$ and, up to interchanging~$x$
with $y$, we assume that~$\dd(y)\le \dd(x)$.
Then there exists~$z\in\partial\Omega$ such that~$|y-z|=\dd(y)$,
and so
\begin{equation}\label{vic}
|u(y)|=|u(y)-u(z)|\le [u]_{C^s(\R^n)}|y-z|^s =[u]_{C^s(\R^n)} \dd^s(y).
\end{equation}
Also, by~\eqref{GN}, \eqref{12sdv56} and~\eqref{cc},
\begin{equation*}
|u(x)-u(y)|\le [u]_{\alpha;\Omega}^{(-s)}\,|x-y|^\alpha \dd^{s-\alpha}(y).
\end{equation*}
Therefore, recalling also~\eqref{dcuxj},
\begin{equation*}
\begin{split}
|w(x)-w(y)|\,&\le |\vartheta(x)|\,|u(x)-u(y)|+|u(y)|\,|\vartheta(x)-\vartheta(y)|\\
&\le C\big(
[u]_{\alpha;\Omega}^{(-s)}\,|x-y|^\alpha \dd^{s-\alpha}(y)+[u]_{C^s(\R^n)} \dd^s(y)
R^{-\alpha} |x-y|^\alpha\big).
\end{split}\end{equation*}
This says that, in this case,
condition~\eqref{L01L} holds true, with~$\ell=
C\,\big( [u]_{\alpha; \Omega}^{(-s)} +[u]_{C^s(\R^n)}\big)$,
and this proves~\eqref{WA} when~$\alpha\in(s,1]$.

Now we prove~\eqref{WA} when~$\alpha\in(1,1+s)$.
In this case, we can write
\begin{equation}\label{dd}
\alpha=1+\alpha', \end{equation}
with~$\alpha'\in(0,s)$,
hence we use~\eqref{GN} (with index~$j=1$) to deduce that in this case
\begin{equation}\label{9sd567897899fg}
\|u\|_{\alpha;\Omega}^{(-s)}\ge
\dd^{1-s}(x)\,|\nabla u(x)|.\end{equation}
Hence, recalling~\eqref{dcuxj2} and~\eqref{vic},
\begin{equation}\label{9idvfb78gfd}
\begin{split}
|\nabla w(x)|\,&\le |\vartheta(x)|\,|\nabla u(x)|+
|\nabla\vartheta(x)|\,|u(x)|\\
&\le C\,\big(\|u\|_{\alpha;\Omega}^{(-s)}\, \dd^{s-1}(x)+
R^{-1}\,[u]_{C^s(\R^n)} \dd^s(x)\big).
\end{split}
\end{equation}
Now we take~$x$, $y\in\Omega$, with~$|x-y|\le R$,
and suppose, without loss of generality that~$\dd(x,y)=\dd(y)\le \dd(x)$.
Since~$\alpha'\in(0,s)$, we have that
\begin{equation}\label{eoi}
|u(x)-u(y)|
\le [u]_{C^s(\R^n)} |x-y|^s\le
[u]_{C^s(\R^n)} R^{s-\alpha'} \,|x-y|^{\alpha'}.
\end{equation}
Also,
using~\eqref{dd} once again, we obtain
$$ [u]_{\alpha;\Omega}^{(-s)}
\ge \dd^{\alpha-s}(y)
\frac{|\nabla u(x)-\nabla u(y)|}{|x-y|^{\alpha'}} $$
and therefore, recalling also~\eqref{dcuxj}, \eqref{dcuxj2}, \eqref{dcuxj3},
\eqref{vic}
\eqref{9sd567897899fg} and~\eqref{eoi},
\begin{eqnarray*}
&& |\nabla w(x)-\nabla w(y)|\\
&=& |
\vartheta(x)\nabla u(x)+\nabla\vartheta(x)u(x)
-\vartheta(y)\nabla u(y)-\nabla\vartheta(y)u(y)|\\
&\le&
|\vartheta(x)\nabla u(x)-\vartheta(x)\nabla u(y)|
+|\vartheta(x)\nabla u(y)
-\vartheta(y)\nabla u(y)|\\ &&\quad
+|\nabla\vartheta(x)u(x)-\nabla\vartheta(x)u(y)|+
|\nabla\vartheta(x) u(y)-\nabla\vartheta(y)u(y)| \\
&\le& C
|\nabla u(x)-\nabla u(y)|
+CR^{-\alpha'} |\nabla u(y)|\,|x-y|^{\alpha'}
\\ &&\quad+CR^{-1}
|u(x)-u(y)|+ [u]_{C^s(\R^n)} \dd^s(y)\,
|\nabla\vartheta(x) -\nabla\vartheta(y)|\\
&\le&
C\,[u]_{\alpha;\Omega}^{(-s)} \dd^{s-\alpha}(y)
|x-y|^{\alpha'}
+CR^{-\alpha'}
\|u\|_{\alpha;\Omega}^{(-s)} \dd^{s-1}(y)
\,|x-y|^{\alpha'}
\\ &&\quad+CR^{s-\alpha'-1}
[u]_{C^s(\R^n)}\,|x-y|^{\alpha'}
+CR^{-1-\alpha'} [u]_{C^s(\R^n)} \dd^s(y)\,|x-y|^{\alpha'}.
\end{eqnarray*}
By taking common factors and recalling~\eqref{dd}, we obtain
that~$|\nabla w(x)-\nabla w(y)|$
is bounded from above by
\begin{equation}\label{0sdfhktfk--1}\begin{split}
& C\,\big( [u]_{C^s(\R^n)}+
\|u\|_{\alpha;\Omega}^{(-s)}\big)\,|x-y|^{\alpha-1}\,
\\&\qquad\cdot \big(
\dd^{s-\alpha}(y)
+R^{1-\alpha} \dd^{s-1}(y)
+R^{s-\alpha}
+R^{-\alpha} \dd^s(y)
\big).
\end{split}\end{equation}
Now we observe that
\begin{eqnarray*} &&\frac{1}{R^{\alpha-1} \dd^{1-s}(y)}<
\frac{1}{\dd^{\alpha-s}(y)}+\frac{1}{R^{\alpha-s}}\\
{\mbox{and }}&&
\frac{1}{R^{\alpha-s}}< \frac{1}{\dd^{\alpha-s}(y)}+
\frac{\dd^s(y)}{R^{\alpha} },\end{eqnarray*}
which follows directly from Young inequality. (Alternatively, this can be checked by considering the cases~$R\ge \dd(y)$
and~$R<\dd(y)$, and recalling that here~$\alpha>1>s$.)
Consequently,
$$ R^{1-\alpha} \dd^{s-1}(y)
+R^{s-\alpha}
\le 2\big(
\dd^{s-\alpha}(y)+R^{-\alpha} \dd^s(y)\big)$$
and then~\eqref{0sdfhktfk--1} yields
that
that~$|\nabla w(x)-\nabla w(y)|$
is bounded from above by
$$ C\,\big( [u]_{C^s(\R^n)}+
\|u\|_{\alpha;\Omega}^{(-s)}\big)\,|x-y|^{\alpha-1}\,
\big(
\dd^{s-\alpha}(y)+R^{-\alpha} \dd^s(y)
\big).$$
This and~\eqref{9idvfb78gfd}
complete the proof of~\eqref{WA} also when~$\alpha\in(1,1+s)$.

Having completed the proof of~\eqref{WA},
now we estimate~$\|w\|_{C^s(\R^n)}$.
We claim that
\begin{equation}\label{w c s}
[w]_{C^s(\R^n)}\le C\,[u]_{C^s(\R^n)}
.\end{equation}
To check this, we fix~$a$, $b\in\R^n$ and
we aim to bound~$|w(a)-w(b)|$.
To this goal, by possibly exchanging $a$ and~$b$,
we suppose that~$|a|\ge|b|$.
Now we distinguish two cases: either~$|b|\ge2R$ or~$|b|<2R$.
If~$|b|\ge 2R$
we have that both~$a$ and~$b$ lie outside~$B_{3R/2}$, and so~$\vartheta(a)=
\vartheta(b)=1$. Accordingly
$$ |w(a)-w(b)|=|u(a)-u(b)|\le [u]_{C^s(\R^n)}|a-b|^s,$$
and~\eqref{w c s} is proved in this case. On the other hand, if~$|b|<2R$
we use~\eqref{5.FF} and we get
$$|u(b)|=|u(b)-u(p_\star)|\le [u]_{C^s(\R^n)}|b-p_\star|^s\le C[u]_{C^s(\R^n)}R^s. $$
Therefore, recalling~\eqref{dcuxj} (used here with~$\alpha'=s$),
\begin{eqnarray*}
|w(b)-w(b)|&\le&
|\vartheta(a)|\,|u(a)-u(b)|+|u(b)|\,|\vartheta(a)-\vartheta(b)|\\
&\le& C\,\Big( [u]_{C^s(\R^n)}|a-b|^s +[u]_{C^s(\R^n)}R^s
\cdot R^{-s} |a-b|^s\Big),
\end{eqnarray*}
which establishes~\eqref{w c s}
in this case too.

Now we take~$\ell$ as in~\eqref{WA}
when~$\alpha\in(s,1+s)$ and,
for definiteness, we also set~$\ell=0$
when~$\alpha
\in[0,s]$: then, with this notation, we deduce
by Lemma~\ref{WW} and formula~\eqref{w c s} that
\begin{eqnarray*}
&& R^{2s}\|Lw\|_{L^\infty(B_{R/2})} \le CR^s\,[w]_{C^s(\R^n)}\le
CR^s\,[u]_{C^s(\R^n)}\\
{\mbox{and }}&&
R^{\alpha+s}[Lw]_{C^\alpha(B_{R/2})} \leq C\,\Big(
[w]_{C^s(\R^n)}+
\ell \Big)\le
C\,\Big(
[u]_{C^s(\R^n)}+\ell\Big),
\end{eqnarray*}
and, in case~$\alpha\in(1,1+s)$, also
$$ R^{1+s}\|\nabla Lw\|_{L^\infty(B_{R/2})}\le C\ell\le
C\,\Big(
[u]_{C^s(\R^n)}+\ell\Big).$$
As a consequence, if~$\alpha=k+\alpha'$ with~$k\in\N$ and~$\alpha'\in(0,1]$,
we can write the weighted estimate
\begin{equation}\label{WWW}
\sum_{j=0}^k R^{j+s} \|D^j Lw\|_{L^\infty(B_{R/2})}
+
R^{\alpha+s}[Lw]_{C^\alpha(B_{R/2})}\le
C\,\Big(
[u]_{C^s(\R^n)}+\ell\Big).
\end{equation}
Now we show that $\bar u\in C^\alpha(\R^n)$ ---recall that $\bar u=u-w$ was defined just before \eqref{dcuxj}---, with
\begin{equation}\label{u bar}
[\bar u]_{C^\alpha(\R^n)}\le C\,\big(
[u]_{C^s(\R^n)}+\ell\big)\, R^{s-\alpha}.
\end{equation}
To this goal,
we distinguish two cases, either~$\alpha\in(0,s]$
or~$\alpha\in(s,1+s)$.

Let us first deal with the case~$\alpha\in(0,s]$.
We let~$x$, $y\in\R^n$ and we estimate~$|\bar u(x)-\bar u(y)|$.
If both~$x$ and $y$ lies outside~$B_{3R/2}$, then~$\bar u(x)=\bar u(y)=0$
and so~$|\bar u(x)-\bar u(y)|=0$. Thus we may assume (up to exchanging~$x$
and~$y$) that~$y\in B_{3R/2}$. So, by~\eqref{5.FF},
\begin{equation}\label{u y 6}
|u(y)|=
|u(y)-u(p_\star)|\le C[u]_{C^s(\R^n)}
\,R^s.\end{equation}
To complete the proof of~\eqref{u bar} when~$\alpha\in(0,s]$,
we now distinguish
two sub-cases: either~$|x|\ge 7R/4$ or~$|x|<7R/4$. If~$|x|\ge 7R/4$,
we have that~$\bar u(x)=0$ and
$$|x-y|\ge |x|-|y| \ge |x|-\frac{3R}{2}\ge\frac{R}{4}.$$
Consequently, we use~\eqref{u y 6}
and we get
\begin{eqnarray*}&& |\bar u(x)-\bar u(y)|=|\bar u(y)|\le |u(y)|\le
C[u]_{C^s(\R^n)}\,R^s \\&&\qquad=C[u]_{C^s(\R^n)}\,R^{s-\alpha} R^\alpha\le
C[u]_{C^s(\R^n)}\,R^{s-\alpha}|x-y|^\alpha,\end{eqnarray*}
and this establishes~\eqref{u bar} in this sub-case.
Now we deal with the sub-case in which~$|x|<7R/4$.
Then~$|x-y|\le (7R/4)+(3R/2)\le CR$
and so
\begin{equation*}
|u(x)-u(y)|\le[u]_{C^s(\R^n)}|x-y|^s
\le C \,[u]_{C^s(\R^n)}|x-y|^\alpha R^{s-\alpha}.\end{equation*}
Therefore, recalling~\eqref{dcuxj} and~\eqref{u y 6},
we deduce that
\begin{eqnarray*}
|\bar u(x)-\bar u(y)|
&\le& |\eta(x)|\,|u(x)-u( y)|+|u(y)|\,|\eta(x)-\eta(y)|\\
&\le& C \,[u]_{C^s(\R^n)}\,R^{s-\alpha}|x- y|^\alpha\end{eqnarray*}
and this proves~\eqref{u bar}
also in this sub-case.

Having completed the proof of~\eqref{u bar}
when~$\alpha\in(0,s]$, we now deal with the case~$\alpha\in(s,1+s)$.
In this case, we take~$x$, $y\in\Omega$
and we write~$\alpha=k+\alpha'$, with~$k\in\{0,1\}$
and~$\alpha'\in(0,1]$ and
\begin{equation}\label{estom}
\ell\ge  \|u\|_{\alpha; \Omega}^{(-s)}
\ge
\dd^{\alpha-s}(x,y)
\frac{|D^ku(x)-D^ku(y)|}{|x-y|^{\alpha'}} +\dd^{-s}(z)\,|u(z)|+
\dd^{k-s}(z)\,|D^k u(z)|,
\end{equation}
for any~$x,y,z\in\Omega$ (notice that we have used~\eqref{GN}
with indices~$j=0$ and~$j=k$).

Also, $\bar u$ (and so~$D^k\bar u$) vanishes in~$\R^n\setminus\overline{B_{3R/2}}$,
therefore,
to estimate~$|D^k\bar u(x)-D^k\bar u(y)|$,
we may assume (up to interchanging~$x$ and~$y$)
that~$y\in \overline{B_{3R/2}}$.

So we distinguish two sub-cases: either also~$x\in
B_{7R/4}$ or~$x\in\R^n\setminus B_{7R/4}$.

Let us start with the sub-case~$x\in B_{7R/4}$.
Notice that
$$ B_{7R/4} \subseteq \big\{
\zeta\in \R^n {\mbox{ s.t. }} \dd(\zeta)\ge R\big\},$$
thanks to~\eqref{5.F}, so in particular~$\dd(x)\ge R$ and~$\dd(y)\ge R$,
and therefore~$\dd(x,y)\ge R$. In addition,
by~\eqref{5.FF}, we have that
$\dd(x)+\dd(y)\le CR$.
So, from~\eqref{estom}, we obtain, for~$x$ and~$y$
as above and every~$\zeta\in B_{7R/4}$,
\begin{equation}\label{IO11}\begin{split}
&|D^ku(x)-D^ku(y)|\le \ell \,\dd^{s-\alpha}(x,y)\,|x-y|^{\alpha'}
\le
C \ell \,R^{s-\alpha}\,|x-y|^{\alpha'} ,\\
&|D^k u(\zeta)|\le \ell \,\dd^{s-k}(\zeta)\le C\ell \,R^{s-k}\\
{\mbox{and }}\;\;&
|u(x)|+|u(y)|\le \ell\,\big(\dd^s(x)+\dd^{s}(y)\big)\le C\ell\,R^s.
\end{split}\end{equation}
We remark that we can also take~$\zeta=x$ or~$\zeta=y$ in~\eqref{IO11}
if we wish.
Now we claim that
\begin{equation}\label{IO12}
|u(x)-u(y)|\le \ell \,R^{s-\alpha'}\,|x-y|^{\alpha'}.
\end{equation}
Indeed, if~$k=0$ we have that~\eqref{IO12}
reduces to~\eqref{IO11}; if instead~$k=1$, we use~\eqref{IO11}
to get that
$$ |u(x)-u(y)|\le \sup_{
\zeta\in B_{7R/4} }
|\nabla u(\zeta)|\,|x-y|\\
\le C\ell \,R^{s-1}\,|x-y|.$$
Also
$$ |x-y|\le |x|+|y|\le \frac74R+\frac32R<4R,$$
thus
$$ |u(x)-u(y)|\le C\ell \,R^{s-1}\,|x-y|^{1-\alpha'}\,|x-y|^{\alpha'}\le
C\ell \,R^{s-1}\,R^{1-\alpha'}\,|x-y|^{\alpha'},$$
up to changing the constants, and this proves~\eqref{IO12}.

Now we remark that
$$ D^k\bar u = D^k(\eta u)= \frac{1}{2-k} \big(\eta D^k u + u D^k \eta\big),$$
for~$k\in\{0,1\}$, therefore
\begin{equation}\label{0eosdythrew11}\begin{split}
&|D^k\bar u(x)-D^k\bar u(y)| \\ =\,&\frac{1}{2-k}
\big| \eta (x) D^k u(x) + u(x) D^k \eta(x)
-\eta (y) D^k u(y)- u(y) D^k \eta(y)\big| \\
\le\,& C\,\Big(
\big| \eta (x) D^k u(x) -\eta (y) D^k u(y)\big|
+ \big|u(x) D^k \eta(x)- u(y) D^k \eta(y)\big|\Big)\\
\le\,&
C\,\Big(|\eta(x)|\,|D^k u(x) -D^k u(y)|
+|D^k u(y)|\,|\eta(x)-\eta (y)|
\\ &\qquad\;+ |u(x)|\,| D^k \eta(x)- D^k \eta(y)|
+|D^k \eta(y)|\,|u(x)-u(y)|
\Big).
\end{split}\end{equation}
Also, by \eqref{dcuxj} and~\eqref{dcuxj3}
$$
|\eta(y)-\eta(x)|
\le CR^{-\alpha'} |x-y|^{\alpha'}\;{\mbox{ and }}\;
|D^k \eta(y)-D^k \eta(x)|
\le CR^{-k-\alpha'} |x-y|^{\alpha'},$$
thus, by~\eqref{IO11} and~\eqref{IO12},
\begin{eqnarray*}
&& |\eta(x)|\,|D^k u(x) -D^k u(y)|
+|D^k u(y)|\,|\eta(x)-\eta (y)|
\\ &&\qquad+ |u(x)|\,| D^k \eta(x)- D^k \eta(y)|
+|D^k \eta(y)|\,|u(x)-u(y)|
\\ &&\qquad\qquad\le C \ell\,\big( R^{s-\alpha}+R^{s-k-\alpha'}\big)
\,|x-y|^{\alpha'}\\ &&\qquad\qquad=C \ell\,R^{s-\alpha}\,|x-y|^{\alpha'},
\end{eqnarray*}
up to renaming constants.
Then, we insert this into~\eqref{0eosdythrew11}
and we obtain
the proof of~\eqref{u bar} in the sub-case~$
x\in B_{7R/4}$.

Now we consider the sub-case~$x
\in\R^n\setminus B_{7R/4}$.
In this case~$\bar u$ (and so~$D^k\bar u$) vanishes in the vicinity
of~$x$, therefore
$$ |D^k\bar u(x)-D^k\bar u(y)|=
|D^k\bar u(y)|=
\frac{1}{2-k} \big|\eta(y) D^k u(y) + u(y) D^k \eta(y)\big|.$$
Therefore, by~\eqref{dcuxj2} and~\eqref{IO11},
\begin{equation}\label{dsocghjkfgs99}
|D^k\bar u(x)-D^k\bar u(y)|\le
C\big(
|\eta(y)|\,| D^k u(y)| + |u(y)|\,| D^k \eta(y)|
\big)\le C\ell R^{s-k}.\end{equation}
Now we use that~$x$ is outside~$B_{7R/4}$ and~$y$ inside~$\overline{B_{3R/2}}$
to conclude that~$|x-y|\ge |x|-|y|\ge R/4$. Hence we deduce from~\eqref{dsocghjkfgs99}
that
$$ |D^k\bar u(x)-D^k\bar u(y)|\le
C\ell R^{s-k-\alpha'}\,R^{\alpha'}\le
C\ell R^{s-\alpha}\,|x-y|^{\alpha'}.$$
This completes the proof of~\eqref{u bar} in the case~$\alpha\in(s,1+s)$.
The proof of~\eqref{u bar} is therefore finished.

Now we show that
\begin{equation}\label{nuova}
\|\bar u\|_{L^\infty(\R^n)}\le C R^s \,\big([u]_{C^s(\R^n)}+\ell\big).
\end{equation}
For this, fix~$x\in\R^n$. If~$|x|\ge 2R$, then~$|\bar u(x)|=0$
and we are done, so we may suppose that~$|x|<2R$.
If~$\alpha\in(0,s]$, we use the fact that~$\bar u(2R e_1)=0$
to conclude that
$$ |\bar u(x)|=|\bar u(x)-\bar u(2R e_1)|\le [\bar u]_{C^\alpha(\R^n)}
|x-2Re_1|^\alpha\le C [\bar u]_{C^\alpha(\R^n)} R^\alpha,$$
hence, by~\eqref{u bar}, we see that~$|\bar u(x)|\le C
\big([u]_{C^s(\R^n)}+\ell\big) R^{s-\alpha}\cdot
R^\alpha$, as desired.
If instead~$\alpha\in(s,1+s)$, we use~\eqref{GN} with index~$j=0$
to see that
$$ \ell\ge \dd^{-s}(x)\,|u(x)|.$$
Hence, since by~\eqref{5.FF} we have that
$$ \dd(x)\le|x-p_\star|\le|x|+|p_\star|\le 5R,$$
we obtain~$|u(x)|\le \ell\,\dd^s(x)\le \ell R^s$.
These considerations prove~\eqref{nuova}.

Now we claim that,
if~$\alpha=k+\alpha'\in(0,1+s)$ with~$k\in\N$ and~$\alpha'\in(0,1]$, then
\begin{equation}\label{YYY}
\sum_{j=0}^k R^{j-s} \|D^j \bar u\|_{L^\infty(\R^n)}
+
R^{\alpha-s} [\bar u]_{C^\alpha(\R^n)}\le
C \,\big([u]_{C^s(\R^n)}+\ell\big).\end{equation}
Indeed, if~$k=0$, formula~\eqref{YYY} follows by combining~\eqref{u bar}
and~\eqref{nuova}. If instead~$k>0$, then necessarily~$k=1$,
since~$k\le\alpha<1+s<2$.
Consequently, using that~$\bar u$ vanishes
outside~$B_{3R/2}$
(thus~$\nabla\bar u(Re_1)=0$)
and~\eqref{u bar}, we obtain
\begin{eqnarray*} && R^{k-s} \|D^k \bar u\|_{L^\infty(\R^n)}
=R^{1-s} \sup_{B_{3R/2}}|\nabla\bar u(x)|
=R^{1-s} \sup_{B_{3R/2}}|\nabla\bar u(x)-\nabla\bar u(Re_1)|
\\ &&\qquad\le R^{1-s} \,[\bar u]_{C^\alpha(\R^n)}
\sup_{B_{3R/2}} |x-Re_1|^{\alpha'}\le C R^{1-s+s-\alpha+\alpha'}\,
C \,\big([u]_{C^s(\R^n)}+\ell\big)\\
&&\qquad=C \,\big([u]_{C^s(\R^n)}+\ell\big).
\end{eqnarray*}
This, \eqref{u bar}
and~\eqref{nuova} then imply~\eqref{YYY} also in this case.
\medskip

After all this (rather technical, but useful)
preliminary work,
we are ready to perform the regularity theory
needed in this setting. For this,
we notice that~$\bar u=u\eta=u(1-\vartheta)=u-w$, therefore
\begin{equation}\label{ZZ}
L\bar u = g-Lw \qquad{\mbox{ in }} B_{3R/2}.\end{equation}
It is now useful to scale~$\bar u$, by setting~$\bar u_R(x)=R^{-2s}\bar u(Rx)$.
We have that~$L\bar u_R(x)=L\bar u(Rx)$.
Therefore by Theorem~1.1(b) in~\cite{SROP},
\begin{equation}\label{athgr09}
\| \bar u_R \|_{C^{\alpha+2s}(B_{1/4})}\le C\,\big(
\| \bar u_R \|_{C^{\alpha}(\R^n)}+\| L\bar u_R \|_{C^\alpha(B_{1/2})}
\big).\end{equation}
To scale back, we notice that, for any~$\beta$, $a>0$,
\begin{eqnarray*}
&& \|D^j \bar u_R\|_{L^\infty(B_a)} =R^{j-2s}\|D^j\bar u\|_{L^\infty(B_{aR})},
\qquad
[\bar u_R]_{C^{\beta}(B_a)}=R^{\beta-2s}[\bar u]_{C^\beta(B_{aR})},\\
&&
\|D^jL\bar u_R\|_{L^\infty(B_a)} =R^j\|D^j L\bar u\|_{L^\infty(B_{aR})}\qquad
{\mbox{and}}\qquad
[L\bar u_R]_{C^{\beta}(B_a)}=R^{\beta}[L\bar u]_{C^\beta(B_{aR})}.\end{eqnarray*}
So, if we write~$\alpha=k +\alpha'$, with~$k\in\N$ and~$\alpha'\in(0,1]$,
we have that
\begin{equation}\label{9tfdfghTYh889055jk}
\begin{split}
\|\bar u_R\|_{C^{\alpha}(B_{a})}\,=\,&
\sum_{j=0}^{k} \|D^j \bar u_R\|_{L^\infty(B_{a})}+
[\bar u_R]_{C^{\alpha}(B_{a})}\\
=\,&
\sum_{j=0}^{k} R^{j-2s}\|D^j \bar u\|_{L^\infty(B_{aR})}+
R^{\alpha-2s}[\bar u]_{C^{\alpha}(B_{aR})}.
\end{split}\end{equation}
Similarly,
if~$\alpha+2s=k_s +\alpha'_s$, with~$k_s\in\N$ and~$\alpha'_s\in(0,1]$,
\begin{equation}\label{Ploa678iJJJ}
\|\bar u_R\|_{C^{\alpha+2s}(B_{a})}
=
\sum_{j=0}^{k_s} R^{j-2s}\|D^j \bar u\|_{L^\infty(B_{aR})}+
R^{\alpha}[\bar u]_{C^{\alpha+2s}(B_{aR})}\end{equation}
and
\begin{equation}\label{9tfdfghTYh889055jk-folL}
\|L \bar u_R\|_{C^{\alpha}(B_{a})}
=
\sum_{j=0}^{k} R^{j}\|D^j L\bar u\|_{L^\infty(B_{aR})}+
R^{\alpha}[L\bar u]_{C^{\alpha}(B_{aR})}.\end{equation}
{F}rom~\eqref{YYY} and~\eqref{9tfdfghTYh889055jk}
we see that
\begin{equation}\label{OIJJijx78}
R^s\,\|\bar u_R\|_{C^{\alpha}(\R^n)}\le
C \,\big([u]_{C^s(\R^n)}+\ell\big).\end{equation}
Also, from~\eqref{ZZ}
and~\eqref{9tfdfghTYh889055jk-folL},
\begin{equation*}
\|L \bar u_R\|_{C^{\alpha}(B_{1/2})}
\le
\sum_{j=0}^{k} R^{j}\|D^j Lw\|_{L^\infty(B_{R/2})}+
R^{\alpha}[Lw]_{C^{\alpha}(B_{R/2})}+\Gamma_g,
\end{equation*}
where
$$\Gamma_g=
\sum_{j=0}^{k} R^{j}\|D^j g\|_{L^\infty(B_{R/2})}+
R^{\alpha}[g]_{C^{\alpha}(B_{R/2})}.$$
Accordingly, from \eqref{WWW},
\begin{equation*}
R^s \|L \bar u_R\|_{C^{\alpha}(B_{1/2})}
\le
C\,\Big(
[u]_{C^s(\R^n)}+\ell+R^s\Gamma_g\Big).
\end{equation*}
This, \eqref{athgr09} and~\eqref{OIJJijx78}
give that
\begin{equation*}
R^s\,\| \bar u_R \|_{C^{\alpha+2s}(B_{1/4})}
\le C\,\Big(
[u]_{C^s(\R^n)}+\ell+R^s\Gamma_g\Big).
\end{equation*}
As a consequence of this and~\eqref{Ploa678iJJJ}, we conclude that
\begin{equation}\label{67890xsusus}
\sum_{j=0}^{k_s} R^{j-s}\|D^j \bar u\|_{L^\infty(B_{R/4})}+
R^{\alpha+s}[\bar u]_{C^{\alpha+2s}(B_{R/4})}
\le C\,\Big(
[u]_{C^s(\R^n)}+\ell+R^s\Gamma_g\Big).
\end{equation}
Now, from~\eqref{set1}, \eqref{set2} and~\eqref{q}, we have
$$ \dd(p,q)=\dd(p)=3R \;{\mbox{ and }}\; \frac{R}{10}\ge|p-q|=|q|,$$
therefore~$\dd(q)\le |q-p_\star|\le 4R$, thanks to~\eqref{5.FF}.
This says that~$\dd(p)$, $\dd(q)$ and~$\dd(p,q)$ are all comparable to~$R$, hence~$R^s\Gamma_g\le
C\|g\|_{\alpha; \Omega}^{(s)}$ and~\eqref{67890xsusus}
gives
\begin{equation}\label{suhj0dc3RG}
\begin{split}
&\sum_{j=0}^{k_s} \Big( \dd(p)^{j-s} |D^j \bar u(p)|+
\dd(q)^{j-s} |D^j \bar u(q)|\Big)+
\dd^{\alpha+s}(p,q)\,\frac{|D^{k_s}\bar u(p)-D^{k_s}\bar u(q)|}{|p-q|^{\alpha'_s}}
\\&\qquad\le C\,\Big(
[u]_{C^s(\R^n)}+\ell+\|g\|_{\alpha; \Omega}^{(s)}\Big).\end{split}\end{equation}
Since~$\bar u=u$ in~$B_{R/4}$, and~$p$ and~$q$ lie in such ball, thanks to~\eqref{q},
we can replace~$\bar u$ with~$u$ in~\eqref{suhj0dc3RG}, and this
establishes~\eqref{TP-fin-0} when~$|p-q|<\dd(p,q)/30$.
\medskip

Let us now suppose that~$|p-q|\geq \dd(p,q)/30$
and let us prove~\eqref{TP-fin-0} in this case.
The proof of this will rely on the fact that we have already
proved~\eqref{TP-fin-0} when~$|p-q|<\dd(p,q)/30$.
We first check that
\begin{equation}\label{PRE-p}
\sum_{j=0}^{k_s} \dd^{j-s}(p)\,|D^j u(p)|
\le C\,\big( \|g\|_{\alpha; \Omega}^{(s)} +
\|u\|_{C^s(\R^n)} +
\chi_{(s,1+s)}(\alpha)\, \|u\|_{\alpha; \Omega}^{(-s)} \big).
\end{equation}
For this we take a sequence of points~$p_m\to p$ as~$m\to\infty$.
Since~$\dd(p)>0$, for~$m$ large we have that~$|p-p_m|\le \dd(p)/100$
and thus
$$ \dd(p_m)\ge \dd(p)-|p-p_m|\ge \frac{99\,\dd(p)}{100}.$$
Therefore~$\dd(p,p_m)\ge 99\,\dd(p)/100$ and thus
$$ |p-p_m|\le \frac{\dd(p)}{100} \le \frac{\dd(p,p_m)}{99}<\frac{\dd(p,p_m)}{30}.$$
Since we have already proved~\eqref{TP-fin-0} in this case, we can
use it at the points~$p$ and~$p_m$ and conclude that
\begin{eqnarray*}
&&\sum_{j=0}^{k_s} \left( \dd^{j-s}(p)\,|D^j u(p)|
+\dd^{j-s}(p_m)\,|D^j u(p_m)|\right)+
\dd^{\alpha+s}(p,p_m)
\frac{|D^{k_s}u(p)-D^{k_s}u(p_m)|}{|p-p_m|^{\alpha'_s}}\\&&\qquad\le
C\,\big( \|g\|_{\alpha; \Omega}^{(s)} +
\|u\|_{C^s(\R^n)} +
\chi_{(s,1+s)}(\alpha)\, \|u\|_{\alpha; \Omega}^{(-s)} \big),
\end{eqnarray*}
which in turn implies~\eqref{PRE-p}.

The same argument used to prove~\eqref{PRE-p}
applied to the point~$q$ instead of~$p$ gives that
\begin{equation}\label{PRE-q}
\sum_{j=0}^{k_s} \dd^{j-s}(q)\,|D^j u(q)|
\le C\,\big( \|g\|_{\alpha; \Omega}^{(s)} +
\|u\|_{C^s(\R^n)} +
\chi_{(s,1+s)}(\alpha)\, \|u\|_{\alpha; \Omega}^{(-s)} \big).
\end{equation}
Now we want to prove that
\begin{equation}
\label{PRE-p-q}
\dd^{\alpha+s}(p,q)
\frac{|D^{k_s}u(p)-D^{k_s}u(q)|}{|p-q|^{\alpha'_s}}\le
C\,\big( \|g\|_{\alpha; \Omega}^{(s)} +
\|u\|_{C^s(\R^n)} +
\chi_{(s,1+s)}(\alpha)\, \|u\|_{\alpha; \Omega}^{(-s)} \big).
\end{equation}
For this, we use the condition~$|p-q|\ge \dd(p,q)/30$
and the fact that
$$\alpha+s=\alpha+2s-s=k_s+\alpha'_s-s$$
to realize that
\begin{eqnarray*}&&\dd^{\alpha+s}(p,q)
\frac{|D^{k_s}u(p)-D^{k_s}u(q)|}{|p-q|^{\alpha'_s}}
\le  C\,
\dd^{k_s+\alpha'_s-s}(p,q)
\frac{|D^{k_s}u(p)-D^{k_s}u(q)|}{\dd^{\alpha'_s}(p,q)}
\\ &&\qquad\le
C\,
\dd^{k_s-s}(p,q)\big(
|D^{k_s}u(p)|+|D^{k_s}u(q)|
\big)\\ &&\qquad\le C\,\dd^{k_s-s}(p)\,|D^{k_s}u(p)|+
C\,\dd^{k_s-s}(q)\,|D^{k_s}u(q)|.\end{eqnarray*}
This estimate, together with~\eqref{PRE-p} and
\eqref{PRE-q} (used here with~$j=k_s$),
establishes~\eqref{PRE-p-q}.

Now, by collecting the estimates in~\eqref{PRE-p},
\eqref{PRE-q} and~\eqref{PRE-p-q}, we
complete the proof of~\eqref{TP-fin-0}
when~$|p-q|\ge \dd(p,q)/30$.
This is the end of the proof of Theorem~\ref{CS}.
\end{proof}

By iterating Theorem~\ref{CS}, and using again the notation in~\eqref{GN},
we obtain:

\begin{cor}\label{IT}
Let~$\beta\in (0,1+s)$.
Let~$\Omega\subset\R^n$
and assume that either~\eqref{HYP1} or~\eqref{HYP2} is satisfied.

Let $u\in C^s(\R^n)$ be a
solution to \eqref{eq}, with~$g\in C^\beta_{\rm loc}(\Omega)$ and~$
\|g\|_{\beta;\Omega}^{(s)}<+\infty$.

Then~$u\in C^{\beta+2s}_{\rm loc}(\Omega)$
and
\begin{equation}\label{8idf6w7euif}
\|u\|_{\beta+2s; \Omega}^{(-s)}
\le C\,\big( \|g\|_{\beta; \Omega}^{(s)} +
\|u\|_{C^s(\R^n)} \big),
\end{equation}
whenever~$\beta+2s$ is not an integer.
\end{cor}

\begin{proof}
First, notice that \eqref{8idf6w7euif} implies that $u\in C^{\beta+2s}_{\rm loc}(\Omega)$.
To prove \eqref{8idf6w7euif}, the rough idea is to iterate Theorem~\ref{CS}
say, starting with H\"older exponent (possibly below)~$s$, to get~$s+2s$,
then~$s+2s+2s$ and so on, till we reach the threshold
imposed by~$\beta$.

To make this argument rigorous we
argue as follows.
If~$\beta\in(0,s]$, then
we use Theorem~\ref{CS} with~$\alpha=\beta$
and we obtain~\eqref{8idf6w7euif}.
Thus, we can assume that
\begin{equation}\label{ser}
\beta\in(s,1+s).\end{equation}
Given~$s\in(0,1)\cap\Q$ (resp., $s\in(0,1)\setminus\Q$),
we fix~$\vartheta\in(0,s)\setminus\Q$ (resp.,~$\vartheta\in(0,s)\cap\Q$).
By construction, for any~$j\in\N$,
we have that~$\vartheta+2js \not\in \Q$, and so in particular
\begin{equation}\label{not integer}
\vartheta+2js \not\in\N.
\end{equation}
We remark that
\begin{equation}\label{ser2}
\vartheta<s<\beta, \end{equation}
thanks to~\eqref{ser}.
We let~$J\in\N$ the largest integer~$j$ for which~$\vartheta+2js\le\beta+2s$.
By construction
\begin{equation}\label{CONSTR}
\vartheta+2Js\in(\beta,\,\beta+2s].
\end{equation}
Furthermore $J\ge1$, due to~\eqref{ser2}.
We also denote by~$C_1>1$ the constant
appearing in~\eqref{TP-fin} and by~$C_2>1$
the one appearing in Lemma~\ref{NORMS}
(these constants were called~$C$ in those statements, but for clarity
we emphasize now these constant by giving to them a special name
and, without loss of generality, we can suppose that they
are larger than~$1$). Let also~$C_\star=2(C_1+C_2)^2$.
We claim that, for any~$j\in\{1,\cdots,J\}$, we have $u\in C^s(\R^n)\cap C^{\vartheta+2js}_{loc}(\Omega)$ and
\begin{equation}\label{induction}
\|u\|_{\vartheta+2js; \Omega}^{(-s)}
\le C_\star^j\,\big( \|g\|_{\beta; \Omega}^{(s)} +
\|u\|_{C^s(\R^n)} \big).\end{equation}
The proof is by induction. First, we use
Theorem~\ref{CS} with~$\alpha=\vartheta\in(0,s)$:
notice that~$\vartheta+2s$ is not an integer, thanks to~\eqref{not integer},
and therefore Theorem~\ref{CS} yields that
\begin{equation}\label{123}
\|u\|_{\vartheta+2s; \Omega}^{(-s)}
\le C_1\,\big( \|g\|_{\vartheta; \Omega}^{(s)} +
\|u\|_{C^s(\R^n)} \big).\end{equation}
Also, in view of~\eqref{ser2} and Lemma~\ref{NORMS},
we have that~$
\|g\|_{\vartheta; \Omega}^{(s)}\le C_2 \|g\|_{\beta; \Omega}^{(s)}$.
By plugging this information into~\eqref{123}
and using that~$C_1$, $C_2>1$, we see that~\eqref{induction}
holds true for~$j=1$.

Now we perform the induction step, i.e.
we suppose that~\eqref{induction} holds true for
some~$j\in\{1,\cdots,J-1\}$ and we prove it for~$j+1$.
For this, we use
Theorem~\ref{CS} with~$\alpha=\vartheta+2js$. We remark
that~$\vartheta+js+2s\not\in\N$, thanks to~\eqref{not integer},
hence Theorem~\ref{CS} applies and it gives that
\begin{equation}\label{trdfashtj}
\|u\|_{\vartheta+2(j+1)s; \Omega}^{(-s)}
\le C_1\,\big( \|g\|_{\vartheta+2js; \Omega}^{(s)} +
\|u\|_{C^s(\R^n)} +\|u\|_{\vartheta+2js; \Omega}^{(-s)}
\big).\end{equation}
Notice also that, by~\eqref{CONSTR},
$$ \vartheta+2js\le \vartheta+2(J-1)s =\vartheta+2Js-2s\le\beta+2s-2s=\beta.$$
This and Lemma~\ref{NORMS} imply that
\begin{equation*}
\|g\|_{\vartheta+2js; \Omega}^{(s)}\le C_2 \|g\|_{\beta; \Omega}^{(s)}.
\end{equation*}
Accordingly, we deduce from~\eqref{trdfashtj} that
$$ \|u\|_{\vartheta+2(j+1)s; \Omega}^{(-s)}
\le C_1\,\big( C_2\|g\|_{\beta; \Omega}^{(s)} +
\|u\|_{C^s(\R^n)} +\|u\|_{\vartheta+2js; \Omega}^{(-s)}
\big).$$
Hence, since by inductive assumption~\eqref{induction} holds true for~$j$,
\begin{equation}
\begin{split}
\|u\|_{\vartheta+2(j+1)s; \Omega}^{(-s)}
\,&\le
C_1\,\Big( C_2\|g\|_{\beta; \Omega}^{(s)} +
\|u\|_{C^s(\R^n)} +
C_\star^j\,\big( \|g\|_{\beta; \Omega}^{(s)} +
\|u\|_{C^s(\R^n)} \big)\Big)\\
&\le
2 C_1\,C_\star^j\,\big( \|g\|_{\beta; \Omega}^{(s)} +
\|u\|_{C^s(\R^n)} \big).
\end{split}\end{equation}
This proves~\eqref{induction} for~$j+1$, and implies that $u\in C^s(\R^n)\cap C^{\vartheta+2(j+1)s}_{loc}(\Omega)$.
The inductive step is thus completed, and we have
established~\eqref{induction}.

Now we observe that
$$ \|u\|_{\beta; \Omega}^{(-s)}\le C_2
\|u\|_{\vartheta+2Js; \Omega}^{(-s)},$$
due to~\eqref{CONSTR}
and Lemma~\ref{NORMS}. Thus, using~\eqref{induction}
with~$j=J$,
\begin{equation}\label{induction-2}
\|u\|_{\beta; \Omega}^{(-s)}
\le C_\star^{J+1}\,\big( \|g\|_{\beta; \Omega}^{(s)} +
\|u\|_{C^s(\R^n)} \big).
\end{equation}
Now we use Theorem~\ref{CS} for the last time,
here with~$\alpha=\beta$. Notice
that~$\beta+2s\not\in\N$, by the assumption of Corollary~\ref{IT}:
consequently Theorem~\ref{CS} can be exploited and we obtain
$$ \|u\|_{\beta+2s; \Omega}^{(-s)}
\le C_\star^{J+1}\,\big( \|g\|_{\beta; \Omega}^{(s)} +
\|u\|_{C^s(\R^n)} +\|u\|_{\beta; \Omega}^{(-s)}
\big).$$
This and~\eqref{induction-2}
imply that we can bound~$\|u\|_{\beta+2s; \Omega}^{(-s)}$
by a constant times~$\|g\|_{\beta; \Omega}^{(s)} +
\|u\|_{C^s(\R^n)}$, and this completes the proof of
Corollary~\ref{IT}.
\end{proof}

With this, we can now complete the proof of the main result:

\begin{proof}[Proof of Theorem~\ref{MAIN}]
By Proposition~4.6 in~\cite{SROP},
we know that~$u\in C^s(\R^n)$, with
\begin{equation}\label{OP}
\| u\|_{C^s(\R^n)}\le C\,\|g\|_{L^\infty(\Omega)}.
\end{equation}
Also, if~$x\in\Omega_\delta$ then $\dd(x)\ge\delta$,
while if~$x\in\Omega$ then~$\dd(x)$ is controlled by the diameter of~$\Omega$.
In particular~$\dd^{j-s}(x)\ge \delta^j /({\rm diam}(\Omega))^s$
for every~$x\in \Omega_\delta$.
{F}rom this,
and recalling also~\eqref{RE:VB67},
we obtain that
\begin{equation*}\begin{split}
&\|u\|_{\beta+2s; \Omega}^{(-s)}\ge c_o\,
\|u\|_{C^{\beta+2s}(\Omega_\delta)}\\
{\mbox{and }}\;&
\|g\|_{\beta; \Omega}^{(s)}
\le C_o\,\|g\|_{C^\beta(\Omega)}
\end{split}\end{equation*}
for some~$c_o$, $C_o>0$, possibly depending on~$\delta$ and~$\Omega$.
Using this,~\eqref{OP} and
Corollary~\ref{IT} we obtain
\begin{eqnarray*}
&& \|u\|_{C^{\beta+2s}(\Omega_\delta)}\le c_o^{-1}
\|u\|_{\beta+2s; \Omega}^{(-s)}
\\&&\qquad
\le c_o^{-1}\,C\,\big( \|g\|_{\beta; \Omega}^{(s)} +
\|u\|_{C^s(\R^n)} \big)\\
&&\qquad
\le c_o^{-1}\,C\,\big( C_o\,\|g\|_{C^\beta(\Omega)}
+C\,\|g\|_{L^\infty(\Omega)} \big).\end{eqnarray*}
The latter term is in turn bounded bounded by a constant,
possibly depending on~$\delta$ and~$\Omega$, times~$\|g\|_{C^\beta(\Omega)}$,
hence the desired result plainly follows.
\end{proof}

\section{Constructing a counterexample}\label{CO:CO:S}

This section is devoted to the construction
of the counterexample of Theorem~\ref{MAIN2}.
For this, we start with an auxiliary lemma that says, roughly speaking, that the operator~$L$
``looses $2s$ derivatives'':

\begin{lem}\label{0sdfuvetryu45t}
Let~$\beta\in(0,1)$ and~$v\in C^{\beta+2s}(\R^n)$. Then~$Lv\in C^\beta(B_1)$
and~$[Lv]_{C^\beta(B_1)}\le C\,[v]_{C^{\beta+2s}(\R^n)}$.
\end{lem}

\begin{proof} Notice that by construction~$\beta+2s\in(0,3)$.
Let~$x$, $y\in B_1$ and
\begin{equation*}
r=|x-y|<2.\end{equation*}
Also, for any fixed~$\omega\in S^{n-1}$ and~$\rho\ge0$, let
$$ w_{\rho,\omega}(x)=2v(x)-v(x+\rho\omega)-v(x-\rho\omega).$$
Then
\begin{eqnarray*}
&& |Lv(x)-Lv(y)|\le I_1+I_2,\quad{\mbox{with }}\\&& I_1 =
\int_0^r
\,d\rho
\int_{S^{n-1}}\,da(\omega)\,
 \frac{
|w_{\rho,\omega}(x)-w_{\rho,\omega}(y)|}{\rho^{1+2s}}\\
&&{\mbox{and }}\quad I_2 =
\int_r^{+\infty}\,d\rho
\int_{S^{n-1}}\,da(\omega)\,
 \frac{|w_{\rho,\omega}(x)-w_{\rho,\omega}(y)|}{\rho^{1+2s}}
.\end{eqnarray*}
To estimate~$I_1$ and~$I_2$,
we first prove that
\begin{equation}\label{prepa}
|w_{\rho,\omega}(x)-w_{\rho,\omega}(y)|
\le \left\{
\begin{matrix}
C\,[v]_{C^{\beta+2s}(\R^n)}
\,|x-y|^{\beta+2s} & \quad {\mbox{ if }}\beta+2s\in(0,1],\\
C\,[v]_{C^{\beta+2s}(\R^n)}\,\rho
\,|x-y|^{\beta+2s-1}
& \quad {\mbox{ if }}\beta+2s\in(1,2],\\
C\,[v]_{C^{\beta+2s}(\R^n)}\,\rho^2
\,|x-y|^{\beta+2s-2}
& \quad {\mbox{ if }}\beta+2s\in(2,3).
\end{matrix}
\right.\end{equation}
To prove this, let us first consider the case~$\beta+2s\in(0,1]$.
In this case, we have that
\begin{equation}
\label{LAMA}
|v(x\pm\rho\omega)-v(y\pm\rho\omega)|\le
[v]_{C^{\beta+2s}(\R^n)}|x-y|^{\beta+2s},\end{equation}
for every~$\rho\geq0$,
and this implies~\eqref{prepa} when~$\beta+2s\in(0,1]$.

If~$\beta+2s\in(1,2]$
we use the Fundamental Theorem of Calculus in the variable~$\rho$
to write
\begin{equation}\label{TF1}
w_{\rho,\omega}(x) =\int_0^\rho \, d\tau \left[
-\nabla v(x+\tau\omega)\cdot\omega +\nabla v(x-\tau\omega)\cdot\omega\right]
.\end{equation}
Notice also that, for any~$\tau\in\R$,
$$ \big| \nabla v(x+\tau\omega)\cdot\omega-\nabla v(y+\tau\omega)\cdot\omega
\big|\le [v]_{C^{\beta+2s}(\R^n)}|x-y|^{\beta+2s-1},$$
since~$\beta+2s\in(1,2]$. This inequality and~\eqref{TF1}
give that
$$ \big|w_{\rho,\omega}(x)-w_{\rho,\omega}(y)\big|
\le \int_0^\rho
2 [v]_{C^{\beta+2s}(\R^n)}|x-y|^{\beta+2s-1}
\, d\tau,$$
that establishes~\eqref{prepa} in this case.

Now we deal with the case~$\beta+2s\in(2,3)$: for this,
we use the Fundamental Theorem of Calculus in the variable~$\rho$
twice in~\eqref{TF1} and we see that
$$ w_{\rho,\omega}(x) =\int_0^\rho \, d\tau \,\int_{\tau}^{-\tau}\,d\sigma\,
D^2_{ij} v(x+\sigma\omega)\,\omega_i\omega_j.$$
Consequently
\begin{eqnarray*}&& \big|w_{\rho,\omega}(x)-w_{\rho,\omega}(y)\big|
\le
\int_0^\rho \, d\tau \,\int_{-\tau}^{\tau}\,d\sigma\,
\Big| D^2_{ij} v(x+\sigma\omega)-D^2_{ij} v(y+\sigma\omega)\Big|
\\ &&\qquad\leq\int_0^\rho \, d\tau \,\int_{-\tau}^{\tau}\,d\sigma\,
[v]_{C^{\beta+2s}(\R^n)} |x-y|^{\beta+2s-2},\end{eqnarray*}
since~$\beta+2s\in(2,3)$, and this establishes~\eqref{prepa} in this case
as well.

The proof of~\eqref{prepa} is thus complete, and now
we show that
\begin{equation}\label{prepa-diff}
|w_{\rho,\omega}(x)-w_{\rho,\omega}(y)|
\le \left\{
\begin{matrix}
C\,[v]_{C^{\beta+2s}(\R^n)}
\,\rho^{\beta+2s} & \quad {\mbox{ if }}\beta+2s\in(0,1],\\
C\,[v]_{C^{\beta+2s}(\R^n)}\,|x-y|
\,\rho^{\beta+2s-1}
& \quad {\mbox{ if }}\beta+2s\in(1,3).
\end{matrix}
\right.\end{equation}
To prove~\eqref{prepa-diff},
we distinguish three cases.
If~$\beta+2s\in(0,1]$, we have that~\eqref{prepa-diff}
follows easily from the fact that~$|v(x)-v(x\pm\rho\omega)|\le
[v]_{C^{\beta+2s}(\R^n)}\,\rho^{\beta+2s}$, and the same for~$y$
instead of~$x$.
If~$\beta+2s\in(1,2]$ we notice that
$$ \nabla w_{\rho,\omega}(x)=
2\nabla v(x)-\nabla v(x+\rho\omega)-\nabla v(x-\rho\omega)$$
and so
the Fundamental Theorem of Calculus in the space variable gives that
\begin{equation}
\label{0-9-0}
\begin{split}
& w_{\rho,\omega}(x)-w_{\rho,\omega}(y)\\ =\,&\int_0^1\,dt\,
\nabla w_{\rho,\omega}(y+t(x-y))\cdot (x-y)\\
=\,&\int_0^1\,dt\,
\Big(
2\nabla v(y+t(x-y))-\nabla v(y+t(x-y)+\rho\omega)-\nabla v(y+t(x-y)-\rho\omega)
\Big)\cdot(x-y).
\end{split}
\end{equation}
Now we notice that
\begin{equation}\label{9sdfg5rggrre}
\big|\nabla v(y+t(x-y))-\nabla v(y+t(x-y)\pm\rho\omega)\big|
\leq[v]_{C^{\beta+2s}(\R^n)}\,\rho^{\beta+2s-1},
\end{equation}
since here~$\beta+2s\in(1,2]$. By inserting~\eqref{9sdfg5rggrre}
into~\eqref{0-9-0} we obtain
$$ |w_{\rho,\omega}(x)-w_{\rho,\omega}(y)|\le
\int_0^1\,dt\,
2[v]_{C^{\beta+2s}(\R^n)}\,\rho^{\beta+2s-1}\,|x-y|,$$
which proves~\eqref{prepa-diff} in this case.

It remains to prove~\eqref{prepa-diff}
when~$\beta+2s\in(2,3)$. In this case, we apply the
Fundamental Theorem of Calculus in the variable~$\rho$
once more to obtain
\begin{eqnarray*}
\partial_i v(y+t(x-y))-\partial_i v(y+t(x-y)\pm \rho\omega) =
\pm \int_\rho^0\,d\sigma\,
D^2_{ij} v(y+t(x-y)\pm \sigma\omega) \,\omega_j
\end{eqnarray*}
and therefore
\begin{eqnarray*}
&& \Big|
2\partial_i v(y+t(x-y))-\partial_i v(y+t(x-y)+\rho\omega)-
\partial_iv(y+t(x-y)-\rho\omega)\Big|\\
&=&\left|
\int_\rho^0\,d\sigma\,
\Big( D^2_{ij} v(y+t(x-y)+\sigma\omega)
-D^2_{ij} v(y+t(x-y)-\sigma\omega)\Big)
\,\omega_j
\right|
\\ &\le& \int_0^\rho\,d\sigma\;
[v]_{C^{\beta+2s}(\R^n)} (2\sigma)^{\beta+2s-2}
\\ &=& C\,[v]_{C^{\beta+2s}(\R^n)}\,\rho^{\beta+2s-1},
\end{eqnarray*}
where the condition~$\beta\in(2,3)$ was used. By plugging
this information into~\eqref{0-9-0}, we obtain
$$ |w_{\rho,\omega}(x)-w_{\rho,\omega}(y)|\le \int_0^1\,dt\,
C\,[v]_{C^{\beta+2s}(\R^n)}\,\rho^{\beta+2s-1}\,|x-y|,$$
which establishes~\eqref{prepa-diff} also in this case.

Now we show that
\begin{equation}\label{prepa2}
{\mbox{if $\beta+2s\in (0,2]$, then }}\quad
|w_{\rho,\omega}(x)|
\le C\,[v]_{C^{\beta+2s}(\R^n)}
\,\rho^{\beta+2s}.\end{equation}
Indeed, if~$\beta\in(0,1]$
we have that~$|v(x)-v(x\pm\rho\omega)|\le [v]_{C^{\beta+2s}(\R^n)}
\,\rho^{\beta+2s}$, and this implies~\eqref{prepa2} in this case.
If instead~$\beta\in (1,2]$, then
we use~\eqref{TF1} to see that
$$ |w_{\rho,\omega}(x)| \le
\int_0^\rho \, d\tau\,
|\nabla v(x+\tau\omega)-\nabla v(x-\tau\omega)|
\le \int_0^\rho \, d\tau \, [v]_{C^{\beta+2s}(\R^n)} \,(2\tau)^{\beta+2s-1},$$
which gives~\eqref{prepa2} also in this case.

Now we claim that there exists~$\kappa \in(0,\,\beta)$ such that
\begin{equation}\label{prepa3}\begin{split}
&{\mbox{for any $\rho\in [0,r]$}}\\
&\big|w_{\rho,\omega}(x)-w_{\rho,\omega}(y)\big|\le
C\,[v]_{C^{\beta+2s}(\R^n)}\,\rho^{2s+\kappa} r^{\beta-\kappa}.
\end{split}\end{equation}
To check this, we
distinguish two cases. When~$\beta+2s\in(0,2]$,
we define~$\kappa=\beta/2$ and
use~\eqref{prepa2}
and the assumption that~$\rho\le r=|x-y|$
to conclude that
$$ |w_{\rho,\omega}(x)|
\le C\,[v]_{C^{\beta+2s}(\R^n)}
\,\rho^{\kappa+2s}\,\rho^{\beta-\kappa}\le
C\,[v]_{C^{\beta+2s}(\R^n)}
\,\rho^{\kappa+2s}\,r^{\beta-\kappa}.$$
Since the same holds when~$x$ is replaced by~$y$,
we have that~\eqref{prepa3} when~$\beta+2s\in
(0,2]$ follows from the above formula
and the triangle inequality.

When instead~$\beta+2s\in(2,3)$, we take
$\kappa=\min\left\{ \frac{\beta}{2},\,1-s\right\}$
and we use~\eqref{prepa} and the assumption that~$\rho\le r=|x-y|$
to deduce that
\begin{eqnarray*}&& \big|w_{\rho,\omega}(x)-w_{\rho,\omega}(y)\big|
\le
C\,[v]_{C^{\beta+2s}(\R^n)}\,\rho^2 r^{\beta+2s-2}
\\ &&\qquad=
C\,[v]_{C^{\beta+2s}(\R^n)}\,\rho^{2s+\kappa} \rho^{2-2s-\kappa}
r^{\beta+2s-2}\le
C\,[v]_{C^{\beta+2s}(\R^n)}\,\rho^{2s+\kappa} r^{2-2s-\kappa}
r^{\beta+2s-2},
\end{eqnarray*}
which provides the proof of~\eqref{prepa3}
also in this case.

Having completed these preliminary estimates, we are now
in the position to
estimate~$I_1$. For this, we use~\eqref{prepa3}
and the fact that~$\kappa>0$ to see that, for any fixed~$\omega\in S^{n-1}$,
$$ \int_0^r
\frac{\big|w_{\rho,\omega}(x)-w_{\rho,\omega}(y)\big|}{\rho^{1+2s}}\,
d\rho\le
C\,[v]_{C^{\beta+2s}(\R^n)}\,
r^{\beta-\kappa}
\int_0^r \rho^{2s+\kappa-1-2s}\,d\rho = C\,[v]_{C^{\beta+2s}(\R^n)}\,
r^{\beta}.$$
As a consequence, by integrating in~$\omega\in S^{n-1}$,
we obtain that
\begin{equation}\label{NA-J2}
I_1\le C[v]_{C^{\beta+2s}(\R^n)}r^\beta.
\end{equation}
Now we estimate~$I_2$.
We claim that
\begin{equation}\label{NA-J3}
I_2\le C[v]_{C^{\beta+2s}(\R^n)}r^\beta.
\end{equation}
To prove this, we distinguish two cases.
If~$\beta+2s\in(0,1]$,
we use~\eqref{LAMA} and the fact that~$|x-y|=r$
to deduce that
$$ \big|w_{\rho,\omega}(x)-w_{\rho,\omega}(y)\big|\le
C\,[v]_{C^{\beta+2s}(\R^n)}\,r^{\beta+2s}.$$
Therefore
$$ I_2\le C\,[v]_{C^{\beta+2s}(\R^n)}\,r^{\beta+2s}
\int_r^{+\infty}\,d\rho
\,\rho^{-1-2s}
=C\,[v]_{C^{\beta+2s}(\R^n)}\,r^{\beta},$$
which proves~\eqref{NA-J3} in this case.

If instead~$\beta+2s\in(1,3)$
we use~\eqref{prepa-diff} to write
$$\big|w_{\rho,\omega}(x)-w_{\rho,\omega}(y)\big|
\le
C\,[v]_{C^{\beta+2s}(\R^n)}\,r
\,\rho^{\beta+2s-1}$$
and so to obtain that
$$ I_2\le C\,[v]_{C^{\beta+2s}(\R^n)}\,r
\int_r^{+\infty}\,d\rho\,\rho^{\beta+2s-1-1-2s}
=C\,[v]_{C^{\beta+2s}(\R^n)}\,r\,r^{\beta-1}.$$
This proves~\eqref{NA-J3} also in this case.

By combining~\eqref{NA-J2} and~\eqref{NA-J3}, we conclude that
$$|Lv(x)-Lv(y)|\le I_1+I_2\leq C\,[v]_{C^{\beta+2s}(\R^n)}r^\beta
=C\,[v]_{C^{\beta+2s}(\R^n)}\,|x-y|^\beta,$$
from which the desired result easily follows.
\end{proof}

Now we recall a useful, explicit barrier:

\begin{lem}\label{phi}
Let~$\phi(x)=(1-|x|^2)^s_+$
and~$L$ be as in~\eqref{RI}. Then
$L\phi(x)=c$, for every~$x\in B_1$, where~$c>0$ is a suitable constant,
only depending on~$n$ and~$s$.
\end{lem}

\begin{proof} Fix~$x=(x',x_n)\in B_1$, and let~$b=\sqrt{1-|x'|^2}$
and~$\xi=x_n/b$. Notice that~$|x'|^2+x_n^2< 1$, thus
\begin{equation}\label{98iuhg45678} |\xi|<\frac{\sqrt{1-|x'|^2}}{b}=1.\end{equation}
Moreover, for any~$\rho\in\R$,
$$ \phi(x+\rho e_n)=(1-|x'|^2-|x_n+\rho|^2)^s_+
=(b^2-|b\xi+\rho|^2)^s_+ = b^{2s} (1-|\xi+b^{-1}\rho|^2)^s_+.$$
As a consequence, writing~$\phi_o(\zeta) =(1-|\zeta|^2)_+^s$, for any~$\zeta\in\R$, and
using the change of variable~$t=b^{-1}\rho$,
\begin{eqnarray*}
(-\partial^2_{n})^s\phi(x) &=&
\int_{\R}\frac{2\phi(x)-\phi(x+\rho e_n)-\phi(x-\rho e_n)}{\rho^{1+2s}}\,d\rho
\\&=& b^{2s}
\int_{\R}\frac{2(1-|\xi|^2)^s_+
-(1-|\xi+b^{-1}\rho|^2)^s_+ -
(1-|\xi-b^{-1}\rho|^2)^s_+
}{\rho^{1+2s}}\,d\rho \\
&=&
\int_{\R}\frac{2(1-|\xi|^2)^s_+
-(1-|\xi+t|^2)^s_+ -
(1-|\xi-t|^2)^s_+
}{t^{1+2s}}\,dt
\\ &=& (-\partial^2_{n})^s\phi_o(\xi).\end{eqnarray*}
Now, we point out that~$\phi_o$ is a function of one variable,
and~$(-\partial^2_{n})^s\phi_o =c_o$, for some~$c_o>0$, see e.g.~\cite{getoor}.
Thus, recalling~\eqref{98iuhg45678}, we have that~$(-\partial^2_{n})^s\phi(x)=c_o$.
By exchanging the roles of the variables, we obtain similarly that
$$ (-\partial^2_{1})^s\phi(x)=(-\partial^2_{2})^s\phi(x)=
\cdots=(-\partial^2_{n})^s\phi(x)=c_o,$$ from which we obtain the desired result.
\end{proof}

With this, we can now construct our counterexample, by considering the planar
domain~$\Omega\supset B_4$
in Figure~A.\medskip

\begin{figure}
\begin{center}
\includegraphics[width=.85\textwidth,height=.85\textheight,keepaspectratio]{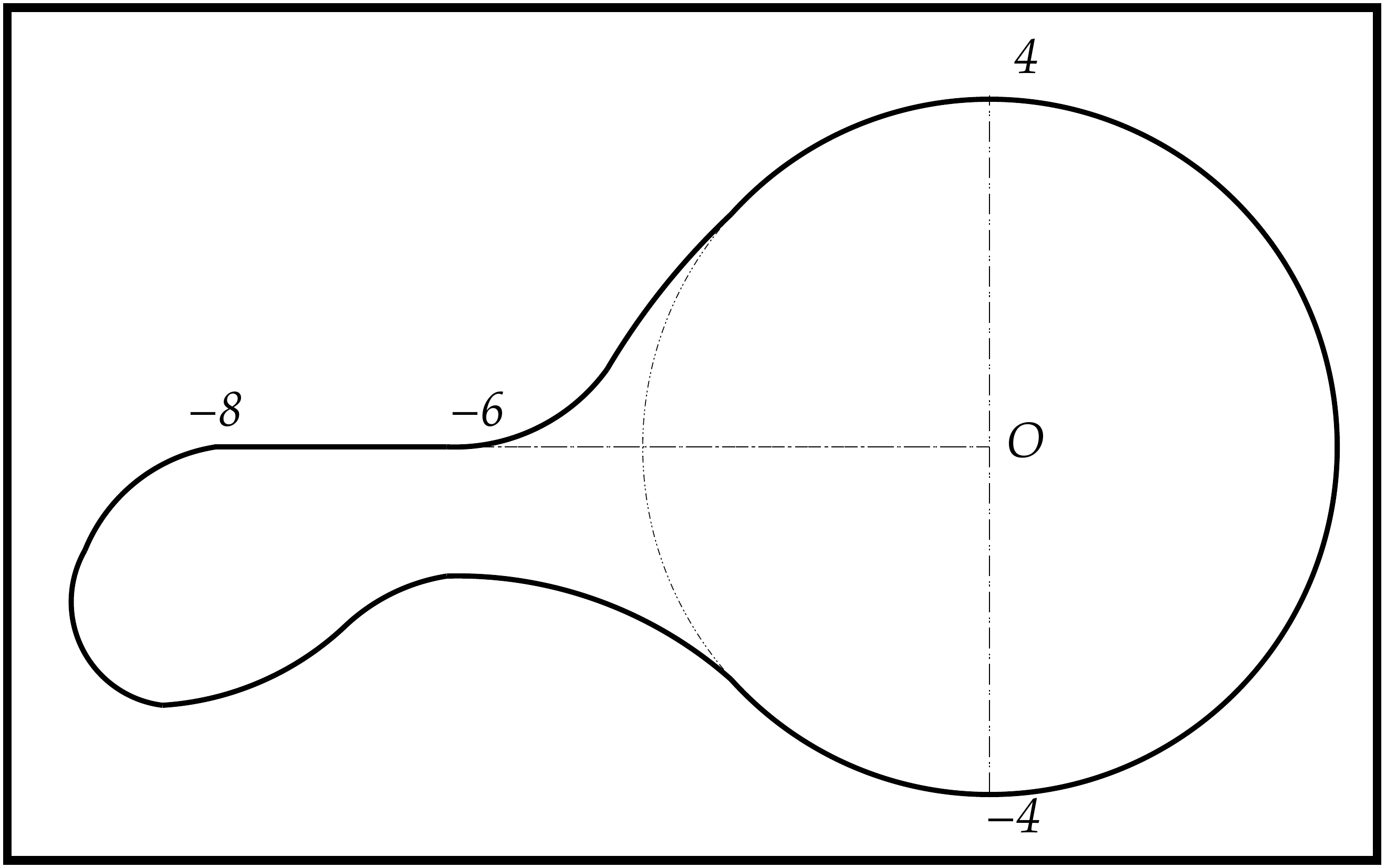}
\\{\footnotesize{{\sc Figure A.} The domain of Theorem~\ref{MAIN2}.}}
\end{center}
\medskip
\end{figure}

\begin{proof}[Proof of Theorem~\ref{MAIN2}]
The fact that~$u\in C^s(\overline\Omega)$ follows from Proposition~4.6
of~\cite{SROP}.
Now suppose, by contradiction, that~$u\in C^{3s+\epsilon}_{\rm loc}(\Omega)$.
Let~$\theta\in C^\infty_0 (B_2)$ with~$\theta\equiv1$ in~$B_1$,
and let~$v=\theta u$ and~$w=u-v$. Then~$v\in C^{3s+\epsilon}(\R^n)$
and so, by Lemma~\ref{0sdfuvetryu45t} (being~$B_1$ there
any ball in~$\R^n$), we have that~$Lv\in C^{s+\epsilon}(\R^n)$.
Hence
\begin{equation} \label{0wdff432gh5trf3edaa}
Lw = 1 -Lv \in C^{s+\epsilon}(\R^n).
\end{equation}
Now, we take~$\eta\in(0,1/2)$ and we evaluate~$Lw$
at the two points~$x_1=(0,0)$ and~$x_2=(0,-\eta)$.
We notice that~$w=(1-\theta)u$, so~$w\equiv0$ in~$B_1$, hence~$w(x_1)=w(x_2)=0$
(as well as~$w(x_2+\rho\omega)=w(x_2-\rho\omega)=0$ if~$\rho\in[0,1/2]$
and~$\omega\in S^{n-1}$), and
$$ Lw(x_1)-Lw(x_2)=
\int_{1/2}^{+\infty}\,d\rho\,
\int_{S^{n-1}}\,da(\omega)\, \frac{w(x_2+\rho\omega)+w(x_2-\rho\omega)
-w(\rho\omega)-w(-\rho\omega)}{\rho^{1+2s}}.$$
More precisely, since~$L$ has the form~\eqref{RI},
the anisotropy takes the form in~\eqref{8isdfv-0bg}
and we obtain that
\begin{equation}
\begin{split}& Lw(x_1)-Lw(x_2)\\&\quad=2\sum_{i=1}^2
\int_{\{|\rho|\ge1/2\}}\,d\rho
\, \frac{w(x_2+\rho e_i)
-w(\rho e_i)}{|\rho|^{1+2s}}
\\&\quad=J_1+J_2,\end{split}\end{equation}
where
\begin{equation*}
\begin{split}
&J_1=2\int_{\{|\rho|\ge1/2\}}
 \frac{w(\rho,-\eta)
-w(\rho,0)}{|\rho|^{1+2s}}\,d\rho
\\ {\mbox{and }}\quad&
J_2=2
\int_{\{|\rho|\ge1/2\}}
 \frac{w(0,\rho -\eta)
-w(0,\rho)}{|\rho|^{1+2s}}\,d\rho
.\end{split}\end{equation*}
Though the integrals $J_1$ and~$J_2$ may look alike at a first glance,
they are geometrically very different: indeed
the integral trajectory of~$J_2$ is transverse to the
boundary (hence the interior regularity will dominate
the boundary effects), while the integral trajectory of~$J_1$
sticks at the boundary (hence it makes propagate
the singularity from the boundary towards the interior).
As an effect of these different geometric behaviors, we will prove
that
\begin{equation}\label{CE-1} J_1\ge C^{-1}\eta^s
\; {\mbox{and }}\;
|J_2|\le C\eta^{s+\epsilon},
\end{equation}
for some~$C>1$ (here~$C$ will denote a quantity, possibly varying
from line to line, which may depend on~$u$, but
which is independent of~$\eta$).
To prove~\eqref{CE-1}
it is useful
to observe that, since~$w$ vanishes outside $\Omega$,
the denominator in the integrands defining~$J_1$ and~$J_2$
is bounded uniformly and bounded uniformly away from zero.
We set~$\beta=2s/(1+2s)$ and we notice that~$\beta\in(0,1)$.
If~$\rho\in [-4,-4+\eta^\beta]\cup[4-\eta^\beta,4]$,
we use that~$w\in C^s(\R^n)$ to obtain that
$$ |w(0,\rho+\eta)-w(0,\rho)|\le C\eta^s,$$
and so
\begin{equation}\label{8iyhjvbbfds}
|J_2|\le C\eta^{\beta+s}
+C \int_{\{|\rho|\in[\frac{1}{2}, 4-\eta^{\beta}]\}}
|w(0,\rho +\eta)-w(0,\rho)|\,d\rho
 .\end{equation}
Here, we have also used that $w(0,\rho)=0$ for $|\rho|>4$.
Also, if~$|\rho|\le 4-
\eta^{\beta}$,
we have that~$\dd\big((0,\rho +\eta),(0,\rho)\big)\ge \eta^{\beta}-\eta\ge
\eta^{\beta}/2$, if~$\eta$ is small enough.
Hence we use Theorem~1.1(b) in~\cite{SROP} (applied here with~$\alpha=s$)
and~\eqref{0wdff432gh5trf3edaa}
to see that
\begin{eqnarray*}
|w(0,\rho +\eta)-w(0,\rho +\eta)|&\le& C\,\big(
\|w\|_{C^s(\R^n)}+\|Lw\|_{C^s(\Omega)}\big)
\,\eta^{3s}
\dd^{-2s}\big( (0,\rho +\eta),(0,\rho)\big)
\\ &\le& C\eta^{3s-2s\beta}\end{eqnarray*}
and therefore
$$\int_{\{|\rho|\in[\frac{1}{2}, 4-\eta^{\beta}]\}}
|w(0,\rho +\eta)-w(0,\rho)| \,d\rho
\le C\eta^{3s-2s\beta}.$$
This and~\eqref{8iyhjvbbfds}
imply that
\begin{equation}\label{8iyhjvbbfds-BIS}
|J_2|\le
C\eta^{\beta+s}+C \eta^{3s-2s\beta}
=C\eta^{\beta+s}.\end{equation}
This estimates~$J_2$. Now we estimate~$J_1$. To this goal, when~$\rho
\in [-8-\sqrt[8]\eta,-8]\cup[-6,-6-\sqrt[8]{\eta}]\cup[4-\sqrt[8]\eta,4]$
we use again that~$w\in C^s(\R^n)$ to
obtain that~$|w(\rho,-\eta)
-w(\rho,0)|\le C\eta$, and so
\begin{equation}\label{J1-1}
\left|\int_{\{\rho \in[-8-\sqrt[8]\eta,-8]\cup
[-6,-6+\sqrt[8]\eta]\cup[4-\sqrt[8]\eta,4]\}}
 \frac{w(\rho,-\eta)
-w(\rho,0)}{|\rho|^{1+2s}}\,d\rho
\right| \le C\eta^{s+\frac{1}{8}}.
\end{equation}
Furthermore, if~$\rho\in [-20,-8-\sqrt[8]\eta]\cup
[-6+\sqrt[8]\eta,4-\sqrt[8]\eta]$
we have that
$$\dd\big( (\rho,-\eta),(\rho,-\eta)\big)\ge \sqrt\eta$$
if~$\eta$ is small enough, and thus, by
Theorem~1.1(b) in~\cite{SROP},
we see that
\begin{eqnarray*}
|w(\rho,-\eta)-w(\rho,0)|&\le& C\,\big(
\|w\|_{C^s(\R^n)}+\|Lw\|_{C^s(\Omega)}\big)
\,\eta^{3s}
\dd^{-2s}\big( (0,\rho +\eta),(0,\rho)\big)
\\ &\le& C\eta^{2s}\end{eqnarray*}
and therefore
\begin{equation}\label{J1-2}
\left|\int_{\{\rho \in[-20,-8-\sqrt[8]\eta]
\cup[-6+\sqrt[8]\eta,4-\sqrt[8]\eta]\}}
 \frac{w(\rho,-\eta)
-w(\rho,0)}{|\rho|^{1+2s}}\,d\rho
\right| \le C\eta^{2s}.
\end{equation}
To complete the estimate on~$J_1$,
in virtue of~\eqref{J1-1}
and~\eqref{J1-2}, it only remains to consider the
case in which~$\rho\in [-8,-6]$.
For this,
if~$\rho\in [-8,-6]$, we have that~$w(\rho,0)=0$,
and we claim that
\begin{equation}\label{TT}
w(\rho,-\eta)\ge c \eta^s,
\end{equation}
for some~$c>0$. To check this, fix~$\rho\in [-8,-6]$.
We notice that there exists~$r_o>0$ (independent of~$\rho$)
such that the ball~$B_{r_o}(\rho,-r_o)$ is tangent
to~$\partial\Omega$ at~$(\rho,0)$.
We consider the function~$\phi$ given by Lemma~\ref{phi}
(used here with~$n=2$),
and we set
$$\underline\phi(x)=\phi_o\big( r_o^{-1} (x-(\rho,-r_o))\big)
=\big(1- r_o^{-2} |x-(\rho,-r_o)|^2\big)_+^s.$$
Exploiting Lemma~\ref{phi} and the Comparison Principle,
we obtain that~$u(x)\ge c \underline\phi(x)$ for any~$x\in
B_{r_o}(\rho,-r_o)$,
for some~$c>0$. So, choosing~$x=(\rho,-\eta)$, we have that
$$ |x-(\rho,-r_o)|^2= |(0,r_o-\eta)|^2 =r_o^2 +\eta^2 -2r_o\eta
\leq r_o^2 -r_o\eta,$$
if~$\eta$ is small enough. So we obtain
$$ u(\rho,-\eta)\ge c
\big(1- r_o^{-2} |x-(\rho,-r_o)|^2\big)_+^s\ge c
\big(1- r_o^{-2} (r_o^2 -r_o\eta)\big)_+^s = c r_o^{-s}\eta^s.$$
Since~$u(\rho,-\eta)=w(\rho,-\eta)$, thanks to the choice of the cutoff
functions,
the latter formula proves~\eqref{TT}, up to renaming constants.

Thus, using~\eqref{TT} (and possibly renaming~$c$
once again), we obtain that
$$ \int_{\{\rho \in[-6,-8]\}}
\frac{w(\rho,-\eta)
-w(\rho,0)}{|\rho|^{1+2s}}\,d\rho=
\int_{\{\rho \in[-6,-8]\}}
\frac{w(\rho,-\eta)}{|\rho|^{1+2s}}\,d\rho
\geq c\eta^s.$$
Consequently, recalling~\eqref{J1-1} and~\eqref{J1-2},
$$ J_1\ge c\eta^s-C\eta^{s+\frac{1}{8}}-C\eta^{2s},$$
that gives~$J_1\ge c\eta^s$, up to renaming constants.
This and~\eqref{8iyhjvbbfds-BIS} complete the proof of~\eqref{CE-1}.

Now, by~\eqref{0wdff432gh5trf3edaa}
and~\eqref{CE-1}, we obtain that
$$ C\ge \frac{Lw(x_1)-Lw(x_2)}{|x_1-x_2|^{s+\epsilon}}\ge
\eta^{-s-\epsilon}\big( J_1-|J_2|\big)\ge \eta^{-s-\epsilon}
\big( C^{-1}\eta^s-C\eta^{s+\epsilon}\big)= \frac{C^{-1}}{\eta^{\epsilon}}-C.$$
This is a contradiction if~$\eta$ is sufficiently small
(possibly in dependence of the fixed~$\epsilon>0$). This shows that~$u$
cannot belong to~$C^{3s+\epsilon}_{\rm loc}(\Omega)$ and so
the construction of the counterexample
in Theorem~\ref{MAIN2} is complete.
\end{proof}

\begin{appendix}

\section{Some basic results about the level sets
of the distance function in~$C^{1,1}$ domains}

The goal of this appendix is to give some ancillary
operational results about the distance function
from the boundary of~$C^{1,1}$ domains.
The topic is of course of classical flavor,
and the literature is rich of results
in even more general settings (see e.g.~\cite{FED}),
but we thought it was useful to have the results
needed for our scope at hand, and with proofs that do
not involve any fine argument from Geometric Measure Theory.

The following result states that a $C^{1,1}$ domain satisfies
an inner sphere condition, in a uniform way.
This is probably well known, but we give the details for the
convenience of the reader:

\begin{lem}\label{WK}
Let~$\kappa,K>0$ and~$\Omega\subset\R^n$ be such that
\begin{equation}\label{A1F}
\Omega\cap B_{2\kappa}=\{ x=(x',x_n)
\in B_{2\kappa} {\mbox{ s.t. }}x_n>h(x')\},\end{equation}
for a $C^{1,1}$ function~$h:\R^{n-1}\to\R$,
with~$\|\nabla h\|_{C^{0,1}(\R^n)}\le K$.

Then, there exists~$C>0$, only depending on~$n$ such that
each point of~$(\partial\Omega)\cap B_\kappa$
is touched from the interior by balls of radius
$$ r=\frac{1}{2}\min\left\{ \kappa,\, K^{-1}\right\}.$$

More explicitly, for any~$p\in (\partial\Omega)\cap B_\kappa$
there exists~$q\in\R^n$ such that~$B_r(q)\subseteq\Omega\cap B_{2\kappa}$
and~$p\in\partial B_r(q)$.

The point $q$ is given explicitly by the formula
$$ q=p-\frac{r\,(\nabla h(p'),-1)}{\sqrt{|\nabla h(p')|^2+1}}.$$
\end{lem}

\begin{proof} Let~$p=(p',p_n)=(p',h(p'))\in(\partial\Omega)\cap B_\kappa$.
By construction~$|p-q|^2=r^2$, hence
\begin{equation}\label{09dsa}
p\in\partial B_r(q).\end{equation}
Moreover, if~$x\in B_r(q)$ then
$$|x|\le |x-q|+|q-p| + |p|< 2r+\kappa\le 2\kappa,$$
hence~$B_r(q)\subseteq B_{2\kappa}$.

Therefore, recalling~\eqref{A1F},
in order to show that~$B_r(q)\subseteq \Omega$,
it suffices to prove that
\begin{equation}\label{TP0}
B_r(q)\subseteq
\{ x=(x',x_n)
\in \R^n {\mbox{ s.t. }}x_n>h(x')\}
\end{equation}
To prove this, let
\begin{equation}\label{TP06789}
x\in B_r(q)\end{equation} and define
$$ \xi(x)=p_n +\nabla h(p')\cdot (x'-p').$$
We claim that
\begin{equation}\label{9dfgnb4e5d}
\xi(x)\le x_n.	
\end{equation}
To check this, we use~\eqref{09dsa}, \eqref{TP06789}
and the convexity of the ball to see that, for any~$t\in(0,1]$,
$$ {B_r(q)}\ni p+t(x-p)=
q+\frac{r\,(\nabla h(p'),-1)}{\sqrt{|\nabla h(p')|^2+1}} + t(x-p).$$
As a consequence
\begin{equation}\label{787778}\begin{split}
r^2 \,&> \left|
\frac{r\,(\nabla h(p'),-1)}{\sqrt{|\nabla h(p')|^2+1}} + t(x-p)\right|^2
\\ &=
r^2 + t^2 |x-p|^2
+
\frac{2rt\,(\nabla h(p'),-1)\cdot(x-p)}{\sqrt{|\nabla h(p')|^2+1}} \\
&=r^2 + t^2 |x-p|^2 +\frac{2rt}{\sqrt{|\nabla h(p')|^2+1}}
\big( \nabla h(p')\cdot (x'-p')-(x_n-p_n)\big) \\
&= r^2 + t^2 |x-p|^2 +\frac{2rt}{\sqrt{|\nabla h(p')|^2+1}}
\big( \xi(x)-x_n\big).
\end{split}\end{equation}
Simplifying~$r^2$ to both the terms, multiplying by~$t^{-1}
\sqrt{|\nabla h(p')|^2+1}$ and taking the limit in~$t$, we deduce that
\begin{eqnarray*}
0 &\ge& \lim_{t\to0^+}
t|x-p|^2  \sqrt{|\nabla h(p')|^2+1}
+2r \big( \xi(x)-x_n\big)\\
&=& 2r \big( \xi(x)-x_n\big).\end{eqnarray*}
This proves~\eqref{9dfgnb4e5d}
and we can now continue the proof of~\eqref{TP0}.

To this goal, we use again~\eqref{787778}, here with~$t=1$,
to observe that
\begin{eqnarray*}
0 &> &
|x-p|^2 +\frac{2r}{\sqrt{|\nabla h(p')|^2+1}}
\big( \xi(x)-x_n\big) \\
&\ge&|x'-p'|^2 +\frac{2r}{\sqrt{|\nabla h(p')|^2+1}}
\big( \xi(x)-x_n\big).
\end{eqnarray*}
As a consequence,
\begin{equation}\label{98vfdccg87}\begin{split}
h(x')-x_n \,=\;& h(x')-h(p') +p_n -x_n\\
\le\;& \nabla h(p')\cdot (x'-p')+\|\nabla h\|_{C^{0,1}(\R^n)} |x'-p'|^2+p_n-x_n
\\ \le\;& \xi(x)-x_n +K|x'-p'|^2\\
<\;& \xi(x)-x_n +\frac{2Kr}{\sqrt{|\nabla h(p')|^2+1}}
\big( x_n-\xi(x)\big)
\\=\;&
\big( x_n-\xi(x)\big)\,\left( \frac{2Kr}{\sqrt{|\nabla h(p')|^2+1}}-1\right)
\end{split}\end{equation}
Furthermore,
$$ \frac{2Kr}{\sqrt{|\nabla h(p')|^2+1} }\le
2Kr\le 1.$$
By inserting this and~\eqref{9dfgnb4e5d} into~\eqref{98vfdccg87}
we conclude that~$h(x')-x_n <0$.
This completes the proof of~\eqref{TP0} and thus of Lemma~\ref{WK}.
\end{proof}

As a consequence of Lemma~\ref{WK}, we obtain that
the level sets of the distance function from a~$C^{1,1}$ domain
are locally a Lipschitz graph:

\begin{cor}\label{CR}
Under the assumptions of Lemma~\ref{WK},
there exists~$\kappa_*\in(0,\kappa]$, $K_*\ge K>0$ only depending
on~$n$, $\kappa$ and~$K$
such that for any~$t\in(0,\kappa_*]$ the level set
\begin{equation}\label{9gbf6thj}
\{ x=(x',x_n)
\in B_{\kappa_*} {\mbox{ s.t. }}\dd(x)=t\}\end{equation}
lies in the graph of a Lipschitz function, whose Lipschitz seminorm is bounded
by~$K_*$.
\end{cor}

\begin{proof} First all all,
we show that for any~$x'\in\R^n$ with~$|x'|\le \kappa_*$
there exists a unique~$x_n(x',t)\in\R$ such that~$\dd(x',x_n(x',t))=t$,
i.e. the level set in~\eqref{9gbf6thj} enjoys a graph property
(as long as~$\kappa_*$ is sufficiently small).

Indeed,
given~$x'$ as above, we consider the point~$p=(x',h(x'))
\in(\partial\Omega)$ given by the graph property
of~$\partial\Omega$. Notice that~$\dd(p)=0$.
Also, by Lemma~\ref{WK} we know that a tubular neighborhood
of width~$r$ lies in~$\Omega$: points on the upper boundary
of this neighborhood stay at distance~$r>\kappa_*\ge t$
from~$\partial\Omega$. Therefore, by the continuity of the
distance function, we find~$\ell(x',t)\ge0$
such that~$\dd(p+\ell (x',t)e_n)=t$. Notice that
$$ p+\ell(x',t) e_n =(p',p_n+\ell(x',t))=(x', h(x')+\ell(x',t))$$
hence we found a point~$x_n(x',t)=h(x')+\ell(x',t)$
with the desired properties.

We remark that the point~$(x',x_n(x',t))$ is unique
in~$B_{\kappa_*}$. Indeed, suppose by contradiction
that~$(x',x_n)$, $(x',x_n+\xi)\in
B_{\kappa_*}$
satisfy~$t=\dd(x',x_n)=\dd(x',x_n+\xi)$, with~$\xi>0$.
Since the gradient of the distance function
agrees with the normal~$\nu$ of the projection~$\pi:\Omega\to\partial\Omega$
in the vicinity of the boundary, we obtain that
\begin{equation}\label{9898}\begin{split}
&0=\dd(x',x_n+\xi)-\dd(x',x_n)=\xi
\int_0^1 \partial_n \dd(x',x'+\tau \xi )\,d\tau\\ &\qquad=
\xi
\int_0^1 \nu_n (\pi(x',x'+\tau \xi ))\,d\tau.\end{split}\end{equation}
Notice that
$$ \nu_n=\frac{1}{\sqrt{|\nabla h|^2+1}}\ge\frac{1}{\sqrt{K^2+1}},$$
thus~\eqref{9898} implies that
$$ 0\ge \frac{\xi}{\sqrt{K^2+1}},$$
which is a contradiction, that shows the uniqueness
of the value~$x_n(x',t)$.

Now we show the Lipschitz property of such graph.
For this we observe that it also follows from
Lemma~\ref{WK} that the distance function in~$B_\kappa$ is semiconcave
(see e.g. Proposition~2.2.2(iii) in~\cite{cannarsa}), namely
there exists~$C>0$, only depending on~$n$, $\kappa$ and~$K$, such that,
for any~$x$, $y\in B_\kappa$ and any~$\lambda\in[0,1]$,
\begin{equation}\label{8sdfdgjhsgaa}
\lambda \dd(x)+(1-\lambda) \dd(y)-\dd(\lambda x+(1-\lambda)y))\le C\lambda
(1-\lambda)|x-y|^2.\end{equation}
Our goal is now to show that, for any~$x,y\in B_{\kappa_*}$,
with~$\dd(x)=\dd(y)=t$,
we can bound~$|x_n-y_n|$ by $K_* |x'-y'|$, for a suitable~$K_*$.
For this, without loss of generality, up to exchanging
the roles of~$x$ and~$y$, we may suppose that
\begin{equation}\label{8sdfdgjhsgaa-2}
x_n\ge y_n.
\end{equation}
So, fixed~$x$ and~$y$ as above,
we let~$z=y-x$ and we
obtain from~\eqref{8sdfdgjhsgaa} that
$$ t=\lambda t+(1-\lambda)t\le
\dd(x+(1-\lambda)z)+ C(1-\lambda)|z|^2.$$
So we set~$\epsilon=1-\lambda\in[0,1]$ and we obtain that
\begin{equation}\label{0dsfddddg}
t\le \dd(x+\epsilon z)+C\epsilon |z|^2.\end{equation}
Let~$X=(X',X_n)\in\partial\Omega$ such that~$t=\dd(x)=|x-X|$.
Then
$$ x= X +
\frac{t\,(-\nabla h(X'),1)}{\sqrt{|\nabla h(X')|^2+1}}$$
and
\begin{eqnarray*}
\dd(x+\epsilon z) &\le& |(x+\epsilon z)- X|\\
&=& \left|\epsilon z
+\frac{t\,(-\nabla h(X'),1)}{\sqrt{|\nabla h(X')|^2+1}}
\right|,
\end{eqnarray*}
and so
$$ \dd^2(x+\epsilon z)\le \epsilon^2 |z|^2+
t^2 +\frac{2\epsilon t\,z\cdot(-\nabla h(X'),1)}{\sqrt{|\nabla h(X')|^2+1}}.$$
By comparing this and~\eqref{0dsfddddg} we obtain
\begin{eqnarray*}
&& C^2\epsilon^2 |z|^4 -2\epsilon t\,C |z|^2
=\big(t-C\epsilon |z|^2\big)^2-t^2\\
&&\qquad
\le \dd^2(x+\epsilon z)-t^2
\le \epsilon^2|z|^2+\frac{2\epsilon t\,z\cdot(-\nabla h(X'),1)}{\sqrt{|\nabla h(X')|^2+1}}.
\end{eqnarray*}
We divide by~$2\epsilon t$ and then take the limit as~$\epsilon\to0^+$,
hence we obtain
\begin{equation}\label{s54ef} -C |z|^2 \le
\frac{z\cdot(-\nabla h(X'),1)}{\sqrt{|\nabla h(X')|^2+1}}.\end{equation}
Recalling~\eqref{8sdfdgjhsgaa-2}, we also have that~$z_n\le0$,
and therefore~\eqref{s54ef} gives that
\begin{equation}\label{0sdfg4edlll} -C |z|^2
\le \frac{|\nabla h(X')|\,|z'|}{\sqrt{|\nabla h(X')|^2+1}}
+
\frac{z_n}{\sqrt{|\nabla h(X')|^2+1}}
\le |z'| -\frac{|z_n|}{\sqrt{|\nabla h(X')|^2+1}}.\end{equation}
Now we observe that~$|z|\le|x|+|y|\le 2\kappa_*$,
hence
$$ C|z|\le C\kappa_*\le \frac{1}{2\sqrt{2(K^2+1)}}$$
if we choose~$\kappa_*$ conveniently small.
Thus we obtain from~\eqref{0sdfg4edlll} that
$$ \frac{|z_n|}{\sqrt{K^2+1}}\le
\frac{|z_n|}{\sqrt{|\nabla h(X')|^2+1}}\le
|z'|+C|z|^2=|z'|+\frac{|z|}{2\sqrt{2(K^2+1)}}.$$
In addition
\begin{eqnarray*}
&& |z|=\sqrt{|z'|^2+|z_n|^2}\le\sqrt{2\max\{ |z'|^2,|z_n|^2\}}
=\sqrt{2{\max}^2\{ |z'|,|z_n|\}}\\ &&\qquad=\sqrt{2}\max\{ |z'|,|z_n|\}
\le \sqrt{2}\big(|z'|+|z_n|\big),\end{eqnarray*}
therefore
$$ \frac{|z_n|}{\sqrt{K^2+1}}\le
|z'|+\frac{|z'|}{2\sqrt{K^2+1}}
+\frac{|z_n|}{2\sqrt{K^2+1}}$$
and so, by taking the latter term to the left hand side,
$$ \frac{|z_n|}{2\sqrt{K^2+1}}\le
|z'|+\frac{|z'|}{2\sqrt{K^2+1}},$$
which establishes the desired Lipschitz property.
\end{proof}

Next is an auxiliary measure theoretic
result that follows from Corollary~\ref{CR}
and the Coarea Formula:

\begin{prop}\label{appa}
Let~$\Omega\subset\R^n$ and~$p\in\Omega$.
Assume that there exist~$\kappa>0$, $N\in\N\cup\{+\infty\}$ and~$K>0$
such that
\begin{equation}\label{CO-L6}
\begin{split}
& {\mbox{$\partial\Omega$
is covered by a
family of balls~$B_\kappa(x_i)$,
with~$i\in\{1,\cdots,N\}$ and~$x_i\in\partial\Omega$,}}\\
&{\mbox{with the property that~$\partial\Omega\cap B_{8\kappa}(x_i)$
lies in a $C^{1,1}$ graph}}\\
&{\mbox{whose $C^{1,1}$ seminorm is bounded by~$K$,}}
\end{split}\end{equation}
for any~$i\in\{1,\cdots,N\}$.

Then, there exist~$\kappa_*\in(0,\kappa)$,
possibly depending on~$\kappa$ and~$K$,
and~$C>0$, possibly depending on~$n$, such that
for any~$\mu \in(0,\kappa_*]$ we have that
\begin{eqnarray*}&& \big|
\big\{ x\in\R^n {\mbox{ s.t. }}
p+x\in \Omega\cap A_{R_1,R_2,P} {\mbox{ and }}
\dd(p+x)\in[0,\mu]\big\}\big|\\&&\qquad
\le C \mu
\,{\mathcal{H}}^{n-1} \big( (\partial\Omega)
\cap A_{R_1-\mu,R_2+\mu,P}
\big),\end{eqnarray*}
for any
annulus~$A_{R_1,R_2,P}=
B_{R_2}(P)\setminus B_{R_1}(P)$, with~$P\in\R^n$,
and~$R_1, R_2>0$ with~$R_2-R_1>2\mu$.
\end{prop}

\begin{proof} We can assume that~$\Omega\cap A_{R_1,R_2,P}
\ne\varnothing$, otherwise we are done.
Also, by possibly translating~$\Omega$, we can suppose that~$p=0$.

We cover~$\partial\Omega$
with a finite overlapping family of balls of radius~$\mu$
centered at points of~$\partial\Omega$, say~$B_{\mu}(y_j)$,
with~$j\in\{1,\cdots,M_\mu\}$, for some~$M_\mu\in\N$.

Notice that each ball~$B_{\mu}(y_j)$ is contained in some~$B_{2\kappa}(x_{i_j})$:
indeed, since~$y_j\in \partial\Omega$,
the covering property implies that there exists~$i_j\in\{1,\dots,N\}$ such
that~$y_j\in B_\kappa(x_{i_j})$; accordingly, if~$q\in B_{\mu}(y_j)$,
then
$$ |q-x_{i_j}|\le |q-y_j|+|y_j-x_{i_j}|< \mu+\kappa\le2\kappa,$$
which says that~$B_{\mu}(y_j)\subseteq B_{2\kappa}(x_{i_j})$.

This implies that we can apply~\eqref{CO-L6}
inside each ball~$B_{\mu}(y_j)$.
As a consequence, by Corollary~\ref{CR}
the level sets of the distance function in~$B_{\mu}(y_j)$
are Lipschitz graphs
with respect to the tangent hyperplane of~$\partial\Omega$
at~$y_j$,
therefore
$$ {\mathcal{H}}^{n-1}\big(
\big\{ x\in \Omega\cap B_\mu(y_j){\mbox{ s.t. }}
\dd(x)=t\big\}\big)\le C\mu^{n-1},$$
for some~$C>0$.
On the other hand~$(\partial\Omega)\cap
B_\mu(y_j)$ is also a $C^{1,1}$ graph with respect
to the tangent hyperplane of~$\partial\Omega$
at~$y_j$ and so
$$ {\mathcal{H}}^{n-1}\big(
(\partial\Omega)\cap
B_\mu(y_j)\big)\ge c\mu^{n-1},$$
for some~$c>0$. By comparing the latter two formulas,
and possibly renaming~$C>0$, we conclude that
\begin{equation}\label{09idfbghfr99tye}
{\mathcal{H}}^{n-1}\big(
\big\{ x\in \Omega\cap B_\mu(y_j){\mbox{ s.t. }}
\dd(x)=t\big\}\big)
\leq C\,
{\mathcal{H}}^{n-1}\big(
(\partial\Omega)\cap
B_\mu(y_j)\big).
\end{equation}
Let us now reorder the indices in such a way the
balls~$B_{\mu}(y_1),\cdots,B_{\mu}(y_{L_\mu})$
intersect the annulus~$A_{R_1,R_2,P}$, for some~$L_\mu\in\N$, $L_\mu\le M_\mu$.
The finite overlapping property of the covering gives that
$$\sum_{j=1}^{L_\mu}
{\mathcal{H}}^{n-1}
\left( (\partial\Omega)\cap
B_\mu(y_j)
\right) \le C\,
{\mathcal{H}}^{n-1}
\left( (\partial\Omega)\cap\left( \bigcup_{j=1}^{L_\mu} B_\mu(y_j)
\right) \right)$$
and so, by set inclusions,
\begin{equation}\label{98io}
\sum_{j=1}^{L_\mu}
{\mathcal{H}}^{n-1}
\left( (\partial\Omega)\cap
B_\mu(y_j)
\right) \le C\,
{\mathcal{H}}^{n-1}
\left( (\partial\Omega)\cap A_{R_1-\mu,R_2+\mu,P}
\right).\end{equation}
Furthermore,
the gradient of the distance function agrees
with the normal of the projection in the vicinity of the boundary
(hence it has modulus~$1$), so
we use the Coarea Formula and~\eqref{09idfbghfr99tye}
to obtain that
\begin{eqnarray*}
&& \big|
\big\{ x\in\Omega\cap A_{R_1,R_2,P}
{\mbox{ s.t. }}
\dd(x)\in[0,\mu]\big\}\big|\\
&\le&\sum_{j=1}^{L_\mu}
\big|
\big\{ x\in\Omega\cap B_\mu(y_j){\mbox{ s.t. }}
\dd(x)\in[0,\mu]\big\}\big|\\
&=& \sum_{j=1}^{L_\mu}\int_{
\big\{ x\in \Omega\cap B_\mu(y_j){\mbox{ s.t. }}
\dd(x)\in[0,\mu]\big\}
} \,dx\\
&=&\sum_{j=1}^{L_\mu} \int_{
\big\{ x\in \Omega\cap B_\mu(y_j){\mbox{ s.t. }}
\dd(x)\in[0,\mu]\big\}
} |\nabla \dd(x)|\,dx\\
&=& \sum_{j=1}^{L_\mu}\int_0^\mu
{\mathcal{H}}^{n-1}\big(
\big\{ x\in \Omega\cap B_\mu(y_j){\mbox{ s.t. }}
\dd(x)=t\big\}\big)\,dt
\\ &\le&
C \sum_{j=1}^{L_\mu} \int_0^\mu
{\mathcal{H}}^{n-1}\big( (\partial\Omega)\cap
B_\mu(y_j)\big)\,dt\\
&\le& C \mu \sum_{j=1}^{L_\mu}
{\mathcal{H}}^{n-1}\big( (\partial\Omega)\cap
B_\mu(y_j)\big).\end{eqnarray*}
This and~\eqref{98io} imply the desired result.
\end{proof}

When condition~\eqref{CO-L6} is fulfilled,
it is also possible to control the surface of~$\partial\Omega$
inside an annulus with the ``correct power'' of the
size of the annulus itself.
A precise statement goes as follows:

\begin{lem}\label{app0987uv}
Let~$\Omega\subset\R^n$ and
assume that there exist~$\kappa>0$, $N\in\N\cup\{+\infty\}$ and~$K>0$
such that $\partial\Omega$
is covered by a
family of balls~$B_\kappa(x_i)$,
with~$i\in\{1,\cdots,N\}$ and~$x_i\in\partial\Omega$,
with the property that~$\partial\Omega\cap B_{8\kappa}(x_i)$
lies in a $C^{1,1}$ graph
whose $C^{1,1}$ seminorm is bounded by~$K$,
for any~$i\in\{1,\cdots,N\}$.

Suppose also that~$\Omega$ is bounded, with diameter less than~$D$.
Then, there exists~$C>0$,
possibly depending on~$\kappa$, $K$ and~$D$, such that
\begin{equation}\label{Goal}
{\mathcal{H}}^{n-1} \big(( \partial\Omega)\cap A_R\big)\le C R^{n-1},
\end{equation}
for any~$R>0$, where~$A_R=B_{8R}\setminus B_R$.
\end{lem}

\begin{proof}
First of all, we show that for any~$r\in (0,\kappa/2]$ and any~$p\in\R^n$,
\begin{equation}\label{erg4twgkjhgfdg}
{\mathcal{H}}^{n-1} \big(( \partial\Omega)\cap B_r(p)\big)\le Cr^{n-1}.
\end{equation}
To prove this, we may suppose that~$( \partial\Omega)\cap B_r(p)\ne
\varnothing$, otherwise we are done. Hence, let~$q\in
( \partial\Omega)\cap B_r(p)$. By assumption, there exists~$i\in\{1,\cdots,
N\}$ such that~$q\in B_\kappa(x_i)$. We observe that
if~$y\in B_r(p)$ then
$$|y-x_i|\le|y-p|+|p-q|+|q-x_i|< r+r+\kappa\le 2\kappa,$$
hence~$B_r(p)\subseteq B_\kappa (x_i)$.

Consequently, $(\partial\Omega)\cap B_r(p)$ lies in a Lipschitz
graph, with Lipschitz seminorm controlled by~$K$
and thus
$$ {\mathcal{H}}^{n-1} \big(( \partial\Omega)\cap B_r(p)\big)
\le \int_{\{x'\in\R^{n-1} {\mbox{ s.t. }} |x'|\le r\}}
\sqrt{K^2+1}\,dx'\le Cr^{n-1},$$
and this proves~\eqref{erg4twgkjhgfdg}.

Now we complete the proof of~\eqref{Goal}. We distinguish two cases:
either~$R\le\kappa/2$ or~$R>\kappa/2$.

If~$R\le\kappa/2$, we cover~$B_8\setminus B_1$ by a
family of balls of radius~$1/4$. By scaling, this provides
a finite number of balls of radius~$R/4$ that cover~$A_R$,
say~$B_{R/4}(q_1),\cdots,B_{R/4}(q_M)$
(notice that~$M$ is a fixed, universal number).
Then, by~\eqref{erg4twgkjhgfdg},
$$ {\mathcal{H}}^{n-1} \big(( \partial\Omega)\cap A_R\big)
\le\sum_{i=1}^M
{\mathcal{H}}^{n-1} \big(( \partial\Omega)\cap B_{R/4}(q_i)\big)
\le C R^{n-1},$$
which proves~\eqref{Goal} in this case.

Thus, we now deal with the case~$R>\kappa/2$.
For this, we consider a non overlapping partition of~$\R^n$
into adjacent (closed) cubes of side~$\kappa/n$
(in jargon, a $\kappa/n$-net of~$\R^n$).
Notice that the number of these cubes
needed to cover~$A_R$ in this case depends on~$\kappa$ (which is fixed
for our purposes, since the constant~$C$ in~\eqref{Goal} is allowed
to depend on~$\kappa$)
but also on~$R$, therefore we need a more careful argument
to bound the number of such cubes that really
play a role in our estimates.
Indeed, we claim that
\begin{equation}\label{NUMBER}\begin{split}&
{\mbox{the number of cubes which intersect
$\partial\Omega$ is bounded}}\\&{\mbox{by
some~$C_o>0$ which depends
only on~$D$ and~$\kappa$.}}\end{split}\end{equation}
To prove this, we may suppose that there is a cube~$Q_*$,
that intersects~$\partial\Omega$, otherwise~\eqref{NUMBER}
is true and we are done. So let~$P_*\in (\partial\Omega)\cap Q_*$.
Let~${\mathcal{F}}_0=\{ Q_*\}$ and~$U_0=Q_*$.
Then, we define~${\mathcal{F}}_1$ the
set all the cubes adjacent to~$U_0=Q_*$,
and we let
$$ U_1 = \bigcup_{Q\in {\mathcal{F}}_0\cup {\mathcal{F}}_1 } Q.$$
Then, we let~${\mathcal{F}}_2$ the set of cubes
adjacent to~$U_1$ and we set
$$ U_2 = \bigcup_{Q\in {\mathcal{F}}_0\cup {\mathcal{F}}_1
\cup {\mathcal{F}}_2} Q,$$
and so on, iteratively, we let~${\mathcal{F}}_{i+1}$ the set of cubes
adjacent to~$U_i$
and
$$ U_{i+1} = \bigcup_{Q\in {\mathcal{F}}_0\cup \cdots
\cup {\mathcal{F}}_{i+1}} Q.$$
Notice that~$U_i$ is a cube of side~$(2i+1)\kappa/n$,
that has~$Q_*$ ``in its center'', that is
$$ {\rm dist}\,(\partial U_i , Q_*)=
\inf_{a\in\partial U_i ,\,b\in Q_*}|a-b|=i.$$
Also, if~$Q_\sharp\in {\mathcal{F}}_i$
is such that~$(\partial\Omega)\cap Q_\sharp\ne\varnothing$,
namely there exists~$P_\sharp \in (\partial\Omega)\cap Q_\sharp$,
then we have that~$|P_\sharp-P_*|\le D$, thanks to the property of
the diameter. Also,
\begin{eqnarray*}&&{\rm dist}\,(P_\sharp, Q_*)\ge {\rm dist}\,(\partial U_{i-1} , Q_*)
=i-1\\
{\mbox{and }}&&\sup_{a\in\partial Q_*}|P_* -a|
\le \kappa.\end{eqnarray*}
Therefore
$$ D\ge |P_\sharp-P_*|\ge i-1-\kappa,$$
hence~$i\le D+1+\kappa$.

This means that the cubes that intersect~$\partial \Omega$
lie in~$U_i$, with~$i\le D+1+\kappa$.
Since~$U_i$ contains~$(2i+1)^n/n^n$ cubes,
the number of cubes of the net which intersect~$\partial \Omega$
is at most~$(2(D+1+\kappa)+1)^n/n^n$, which proves~\eqref{NUMBER}.

Furthermore, if~$Q$ is a cube of the family,
we have that~$Q$ is contained in the ball of radius~$\kappa$
with the same center of~$Q$: hence, by~\eqref{erg4twgkjhgfdg},
$$ {\mathcal{H}}^{n-1} \big(( \partial\Omega)\cap Q\big)\le C\kappa^{n-1}.$$
{F}rom this and~\eqref{NUMBER}, we obtain that
\begin{eqnarray*}
&& {\mathcal{H}}^{n-1} \big(( \partial\Omega)\cap A_R\big)
\le \sum_{Q {\mbox{ s.t. }} ( \partial\Omega)\cap Q\ne\varnothing }
{\mathcal{H}}^{n-1} \big(( \partial\Omega)\cap Q\big)
\\ &&\qquad\le \sum_{Q {\mbox{ s.t. }} ( \partial\Omega)\cap Q\ne\varnothing }
C\kappa^{n-1}\le
C_o C\kappa^{n-1}.\end{eqnarray*}
Then, since we are assuming in this case that~$R>\kappa/2$,
$$
{\mathcal{H}}^{n-1} \big(( \partial\Omega)\cap A_R\big)\le 2^{n-1}
C_o C R^{n-1},$$
which proves~\eqref{Goal} also in this case, up to renaming constants.
\end{proof}

\end{appendix}

\end{document}